\documentclass[11pt,a4paper]{article}
\usepackage{hyperref}
\usepackage[pdftex]{graphicx}
\usepackage{amssymb,amsthm,amsmath,hyperref,txfonts}

\usepackage{mathtools}

\usepackage[dvips]{color}
\usepackage{graphicx,float,xcolor}
\usepackage{epstopdf}
\usepackage{tikz}
\usetikzlibrary{positioning,decorations,arrows,patterns,shapes,backgrounds}

\newtheorem{theorem}{Theorem}[section]
\newtheorem{remark}{Remark}[section]

\newtheorem{lemma}[theorem]{Lemma}

\def\XXint#1#2#3{{\setbox0=\hbox{$#1{#2#3}{%
\int}$ }
\vcenter{\hbox{$#2#3$ }}\kern-.6\wd0}}

\def\edc{\end{document}}
\numberwithin{equation}{section}

\date{\today}

\allowdisplaybreaks

\begin{document}

\title{\bf Global well-posedness of the nonhomogeneous incompressible Navier-Stokes-Cahn-Hilliard system with Landau Potential
}
\author{Nie, Rui\thanks{E-mail: 202220798@stumail.nwu.edu.cn}\ ; Fang, Li\thanks{E-mail: fangli@nwu.edu.cn}\thanks{Corresponding author} \\
\textit{\small Department of Mathematics and CNS, Northwest University, Xi'an, P. R. China}\\
Guo, Zhenhua\thanks{zhguo@nwu.edu.cn}\\
          \textit{\small School of Mathematics and Information, Guangxi University, Nanning 530004, China.}
}
\date{}
\maketitle
\begin{abstract}
A diffuse-interface model that describes the dynamics of nonhomogeneous incompressible two-phase viscous
flows is investigated in a bounded smooth domain in ${\mathbb R}^3.$ The dynamics of the state variables is described by the nonhomogeneous incompressible Navier-Stokes-Cahn-Hilliard system.  We first give a blow-up criterion of local strong solution to the initial-boundary-value problem for the case of initial density away from zero. After establishing some key a priori with the help of the Landau Potential, we obtain the global existence and decay-in-time of strong solution, provided that the initial date
$\|\nabla u_0\|_{L^{2}(\Omega)}+\|\nabla \mu_0\|_{L^{2}(\Omega)}+\rho_0$ is suitably small.
\vspace{4mm}\\
\noindent{\textbf{Keywords:} Navier-Stokes-Cahn-Hilliard system, Blow-up criterion,  Global strong solutions, Initial-boundary-value problem, Landau potential}

\noindent{\textbf{AMS Subject Classification (2020):} 35A01, 35Q35,76D03}
\end{abstract}

\section{Introduction}

The Navier-Stokes equations for the nonhomogeneous incompressible fluid  coupled nonlinearly with the convective Cahn-Hilliard equation for the order parameter, describes the motion of an isothermal mixture of two immiscible and nonhomogeneous incompressible fluids subject to phase separation. The order parameter represents the difference of the relative concentrations of the two fluids. The model is known as  Navier-Stokes-Cahn-Hilliard system (see \cite{2008-FE,1998-AMW,Gurtin-1996,Heida-2013,Heida-2012}). In this system, the fluid's velocity field and phase field (concentration field) are mutually coupled, the flow of the fluid affects the evolution of the phase, while changes in the phase influence the movement of the fluid. The primary research objective is to simulate and predict fluid and phase transition phenomena in real physical systems, such as bubble formation and droplet separation. Overall, studying the Navier-Stokes-Cahn-Hilliard system not only helps to deepen the understanding of complex two-phase flow phenomena but also advances the development of application technologies in related fields. One of the nonhomogeneous incompressible  Navier-Stokes-Cahn-Hilliard system is the following
\begin{equation}\label{1.1}
	\begin{cases}
		\rho_t+u\cdot\nabla \rho=0,\\
		\rho u_t +\rho(u\cdot\nabla )u-\operatorname{div}(v(\phi)\mathbb{D}u)+\nabla  p=-\operatorname{div}(\nabla \phi\otimes\nabla \phi),\\
		\operatorname{div}u=0,\\
		\rho \phi_t +\rho u\cdot\nabla \phi=\Delta \mu,\\
		\rho\mu=-\Delta \phi+\rho\Psi^{\prime}(\phi)
	\end{cases}
\end{equation}
in $\Omega\times (0,\infty).$ Here, $\Omega$ is a bounded domain in ${\mathbb R}^3$ with a sufficiently smooth boundary. The unknown function $\rho=\rho(x,t),$ $u=u(x,t)$ and $p=p(x,t)$ denote the density, velocity and pressure of the mixture, respectively. $\phi=\phi(x,t)$ and $\mu=\mu(x,t)$ stand for the difference of fluids concentrations and the chemical potential, respectively.
The quantity $v(\phi)$ denotes the kinematic viscosity of the fluid. As usual, the symmetice $\mathbb{D}u$ denotes the deformation tensor, i.e.
\begin{equation*}
	\mathbb{D}u=\frac12[\nabla u+(\nabla u)^t].
\end{equation*}
The term $\nabla \phi$ is known as Korteweg force. The physical-relevant assumption on the function $\Psi$ is that it must have a double-well structure, each of them representing the two phases of the fluid. Physician \cite{Chella-1996,Doi-1997,Gunton-1983} often propose to consider either the logarithmic potential introduced by Cahn and Hilliardin in \cite{1958-Cahn}
$$\Psi(s) = \frac{\theta}{2}[(1+s)\ln(1+s)+(1-s)\ln(1-s)]-\frac{\theta_0}{2} \ \ \ (\forall s\in[-1,1]),$$
where $0<\theta<\theta_0,$ or the Landau potential deviated by Landau and Lifshitz in  \cite{1968-Landau}
\begin{equation}\label{1.2}
	\Psi(s)=\frac{1}{4}(s^{2}-1)^{2} \ \ \  (\forall s\in \mathbb{R}).
\end{equation}
We consider the system \eqref{1.1}-\eqref{1.2} subjected to the classical boundary conditions
\begin{equation}\label{1.3}
	u=0, \ \ \partial_n\mu=\partial_n\phi=0 \ \ \ \rm{on} \ \ \partial\Omega\times(0,T),
\end{equation}
and the initial conditions
\begin{equation}\label{1.4}
	\rho|_{t=0}=\rho_0(x), \ \ \rho u|_{t=0}=m_0(x), \ \ \phi|_{t=0}=\phi_0(x) \ \ \ \rm{in} \ \ \Omega.
\end{equation}
Here, $n$ is the unit outward normal vector on $\partial\Omega,$ and $\partial_n$ denotes the outer normal derivative on $\partial\Omega.$

The important progresses have been made on the well-posedness and dynamic behaviors of the solutions to the Navier-Stokes-Cahn-Hilliard system \eqref{1.1} (see \cite{2009-AH,Abels-2009,2012-AH,2013-AH15,2008-FE,Boyer-1999,2019-CL,Chen-He-Mei-Shi-2018}) and the references therein. For the logarithmic potential, Abels investigated the incompressible Navier-Stokes-Cahn-Hilliard system in a smooth bounded domain with unmatched viscosities firstly in \cite{Abels-2009}, where the existence of global weak (physical) solutions and the existence of strong solutions (global in two dimensions and local in three dimensions) have been given. Later, Gui and Li proved the global well-posedness of the Cauchy problem of the two-dimension incompressible Navier-Stokes-Cahn-Hilliard system with periodic domain by using energy estimates and the logarithmic Sobolev inequality in \cite{Gui-2018}.  Giorgini, Miranville and Temam \cite{Giorgini-2019}  showed the uniqueness of weak solutions and the existence and uniqueness of global strong solutions to the incompressible Navier-Stokes-Cahn-Hilliard system in a two-dimensional  bounded smooth domain, where the initial velocity satisfied some condition. 
The author considered a singular potential like the logarithmic potential in \cite{He-2021}, which is physically relevant for applications in materials science. By a semi-Galerkin approach, she proved the existence of global weak solutions in a two-dimension/three-dimension smooth bounded domain as well as the uniqueness of global weak solutions in a two-dimension smooth bounded domain. In \cite{He-Wu-2021}, He and Wu established the global strong well-posedness for arbitrarily large and regular initial data in the two-dimension smooth bounded domain.

For the case of Landau potential, we refer the reader to \cite{Boyer-1999,Cao-2012,Gal-2010,Gal-2016,Giorgini-2019}.
Under the condition of allowing diffusion coefficient degradation, Boyer proved the existence and uniqueness of solutions to the incompressible Navier-Stokes-Cahn-Hilliard system under shear through order parameter formulation for two dimensions and three dimensions in \cite{Boyer-1999}.
Cao and Gal obtained the global regularity of strong solutions for the Navier-Stokes-Cahn-Hilliard system in two dimension with mixed partial viscosity and mobility in \cite{Cao-2012}, where the authors also proved the global existence and uniqueness of classical solutions.
In \cite{Lam-Wu-2018}, Lam and Wu derived a class of thermodynamically consistent Navier-Stokes-Cahn-Hilliard system for two-phase fluid flows with density contrast, based on a volume averaged velocity. Moreover, the authors proved the existence of global weak solutions to the incompressible Navier-Stokes-Cahn-Hilliard system on a smooth bounded domain in two and three dimensions, and the existence of a unique global strong solution in two dimensions under some technical assumption on the coefficients.

Recently, Giorgini and Temam \cite{Giorgini-2020} considered the case of initial density away from zero and concentration-depending viscosity with free energy potential equal to either the Landau potential or the logarithmic potential. The authors proved the existence of weak solutions to the nonhomogeneous incompressible Navier-Stokes-Cahn-Hilliard system in a two-dimension/three-dimension bounded smooth domain. Moreover, they showed that the existence of strong solutions are local-in-time in three dimensions and global-in-time in two dimensions with bounded and strictly positive density. The question is whether such solutions can be globally well defined. The motivation of this paper is to establish a blow-up criterion for such strong solution only in terms of the velocity and the global well-posedness of the strong solutions for the initial-boundary problem \eqref{1.1}-\eqref{1.4} in a bounded smooth domain in ${\mathbb R}^3,$ then we characterize the long-time behavior of its global strong solutions as $t\rightarrow \infty.$

There are several works (\cite{Cho-2004,Choe-2003,Fan-2008,Huang-2009,Huang-2009-1,Huang-2011,Huang-2010,Sun-Zhang-2018}) on the blow-up criteria for local strong solutions to the compressible Navier-Stokes equations. Fan and Jiang proved in \cite{Fan-2008} for the two-dimension case, if $7\mu>9\lambda,$ then
$$\lim\limits_{T\rightarrow T^*}\big(\sup\limits_{0\leqslant t\leqslant T}\|\rho\|_{L^\infty(\Omega)}+\int_0^T(\|\rho\|_{W^{1,q_0}(\Omega)}+\|\nabla\rho\|_{L^2(\Omega)}^4)   dt\big)=\infty,$$
where $T^*<\infty$ is the maximal time of existence of strong solutions and $q_0>3$ is a constant. Later, Huang \cite{Huang-2009} established a blow-up criterion for the strong solution to compressible Navier-Stokes equations in three dimension under the extra condition on the viscosity coefficients of $7\mu>\lambda$, that is
\begin{equation}\label{1.5}
	\lim\limits_{T\rightarrow T^*}\int_0^T\|\nabla u\|_{L^\infty(\Omega)}   dt=\infty,\end{equation}
which was analogous to the Beal-Kato-Majda criterion for the ideal incompressible flows in \cite{Beal-1984}. For the initial density away from vacuum, that is
$$\inf\limits_{x\in{\mathbb R}^3}\rho_0(x)>0.$$
Huang and Li removed the crucial condition $7\mu>\lambda$ in \cite{Huang-2009-1} and established a blow-up criterion \eqref{1.5} under the physical restrictions $\mu>0$ and $\lambda+\frac{2}{3}\mu\geqslant 0.$ Moreover, Huang, Li and Xin  improved the result to the following blow-up criterion
$$\lim\limits_{T\rightarrow T^*}\int_0^T\|{\mathbb D} u\|_{L^\infty(\Omega)}dt=\infty$$
by allowing vacuum states initially in \cite{Huang-2011}. Especially, Huang and Li and Xin \cite{Huang-2010} proved the following blow-up criterion to the compressible Navier-Stokes equations
\begin{equation}\label{1.5-1}
	\lim\limits_{T\rightarrow T^{*}}
	\left(\|\operatorname{div}u\|_{L^{1}(0,T;L^{\infty}(\mathbb{R}^{3}))}+\|\rho^{\frac{1}{2}}u\|_{L^{s}(0,T;L^{r}(\mathbb{R}^{3}))}\right)=\infty
\end{equation}
and
\begin{equation}\label{1.5-2}
	\lim\limits_{T\rightarrow
		T^{*}}\left(\|\rho\|_{L^{\infty}(0,T;L^{\infty}(\mathbb{R}^{3}))}+\|\rho^{\frac{1}{2}}u\|_{L^{s}(0,T;L^{r}(\mathbb{R}^{3}))}\right)=\infty,
\end{equation}
with $\frac{2}{s}+\frac{3}{r}\leqslant 1$ and $r$ satisfies $3<r\leqslant\infty.$ Later, Sun and Zhang established the blow-up mechanism of local strong solutions to  compressible Navier-Stokes(-Fourier) system.

Motivated by \cite{Huang-2010}, we give a blow-up criterion for the initial-boundary problem \eqref{1.1}-\eqref{1.4} at first. Before we introduce our main result, let us summarize the principle hypotheses and the notations used through out this paper. We suppose that $\Omega\subset {\mathbb R}^3$ is a bounded domain with $C^3-$boundary. The viscosity $v=v(s)\in W^{1,\infty}({\mathbb R})$ is such that
\begin{equation}\label{1.8}
	0<v_*\leqslant  v(s)\leqslant  v^{*}<\infty \mbox{ for all } s\in {\mathbb R},
\end{equation}
where $v_*$ and $v^{*}$ are two strictly positive values. We denote by $C_{0,\sigma}^{\infty}(\Omega)$ the space of divergence free vector fields in $C_0^{\infty}(\Omega)$. We define $H_{\sigma}(\Omega)$ and $V_{\sigma}(\Omega)$ as the closure of $C_{0,\sigma}^{\infty}(\Omega)$ with respect to the $L^{2}(\Omega)$ and $H_{0}^{1}(\Omega)$ norms, respectively. We also use $(\cdot,\cdot)$ and $\|\cdot\|_{L^{2}(\Omega)}$ for the inner product and norm in $H_{\sigma}(\Omega)$ . The space $V_{\sigma}(\Omega)$ is endowed with the inner product and norm $(u,v)_{V_{\sigma}(\Omega)}=(\nabla u,\nabla v)$ and $\|u\|_{V_{\sigma}(\Omega)}=\|\nabla u\|_{L^{2}(\Omega)}.$ 

The blow-up criterion for the local strong solution to the initial-boundary problem \eqref{1.1}-\eqref{1.4} is stated as follows.

\begin{theorem}\label{thm 1.1}
	Let $\Omega$ be a bounded smooth domain in $\mathbb{R}^{3}$ with boundary $\partial \Omega$ uniformly of class $C^3.$ Assume that
	the hypothesi \eqref{1.8} is satisfied and the initial value $(\rho_0,u_0,\phi_0,\mu_0)$ satisfies
	\begin{equation}\label{1.9}
		\begin{cases}
			0<\rho_*\leqslant \rho_0\leqslant \rho^{*}<\infty, \ u_0\in V_{\sigma}(\Omega), \ \phi_0\in H^{2}(\Omega), \\
			\mu_0=-\frac{\Delta \phi_0}{\rho_0}+\Psi^\prime(\phi_0)\in H^{1}(\Omega),\ \partial_n\phi_0=0\mbox{ on }\partial\Omega.
		\end{cases}
	\end{equation}
	Let  $(\rho,u,\phi,\mu)$ be a strong solution to the initial-boundary problem \eqref{1.1}-\eqref{1.4} on $[0,T^*]$ satisfying
	\begin{equation}\label{1.10}
		\begin{cases}
			\rho\in \mathnormal{C}([0,T^*];L^{q}(\Omega))\cap L^{\infty}(\Omega\times (0,T^*))\cap L^{\infty}(0,T^*;H^{-1}(\Omega)),\\
			u\in \mathnormal{C}([0,T^*];V_\sigma(\Omega))\cap L^{2}(0,T^*;H^{2}(\Omega))\cap H^{1}(0,T^*;H_\sigma(\Omega)),\\
			\phi\in \mathnormal{C}([0,T^*];(W^{2,6}(\Omega))_{w})\cap H^{1}(0,T^*;H^{1}(\Omega)),\\
			\mu\in L^{\infty}(0,T^*;H^{1}(\Omega))\cap L^{2}(0,T^*;W^{2,6}(\Omega))
		\end{cases}
	\end{equation}
	for any $q\in[2,\infty],$ where $T^{*}\in(0,+\infty)$ is the maximum time of existence. Then
	\begin{equation}\label{1.11}
		\lim_{T\to T^{*}}\|u\|_{L^{\frac{4r}{r-6}}(0,T;L^{r}(\Omega))}=\infty
	\end{equation}
	for any $r>6.$
\end{theorem}

\begin{remark}\label{rem-1}
	Compared with \eqref{1.5-1} and \eqref{1.5-2}, the blow-up criterion \eqref{1.11} for the nonhomogeneous incompressible  Navier-Stokes-Cahn-Hilliard system is slightly different in the following twofolds. On one hand, the blow-up criterion \eqref{1.11} is precise, compared to the Serrin's criterion on the Leray-Hopf weak solutions to the three-dimension incompressible Navier-Stokes equations. On the other hand, due to the difference of fluids concentrations and the consideration of the three-dimensional domain and the choice of the potential as the Landau potential $\eqref{1.2}$, a more advanced complexity estimation will be required during the computation process, therefore, it seems impossible to replace the blow-up criterion \eqref{1.11} by  $\|u\|_{L^{s}(0,T;L^{r}(\Omega))}$ with $\frac{2}{s}+\frac{3}{r}\leqslant 1$ and $3<r\leqslant \infty.$
\end{remark}
The next aim of this paper is to establish the global existence of strong solutions to the initial-boundary problem \eqref{1.1}-\eqref{1.4} under some smallness assumptions on the initial data. We prove the existence of global strong solutions to the initial-boundary problem \eqref{1.1}-\eqref{1.4} subject to regular initial data \eqref{1.9} satisfying some suitable bounded condition (see Theorem \ref{thm1.2} blow). The proof is based on the  regularity criterion and an abstract bootstrap argument.

\begin{theorem}\label{thm1.2}
	Let $\Omega$ be a bounded smooth domain in $\mathbb{R}^{3}$ with boundary $\partial \Omega$ uniformly of class $C^3$ and $T>0.$
	Suppose that the hypothesi \eqref{1.8} is satisfied and the initial value $(\rho_0,u_0,\phi_0,\mu_0)$ satisfies \eqref{1.9}. Assume that
	there is a positive number $\varepsilon_0$ depending only on $|\Omega|,v_*,v^*,\|\phi_0\|_{L^2(\Omega)}, \|\mu_0\|_{L^2(\Omega)},$ such that
	\begin{equation}\label{1.12}
		\|\nabla u_0\|_{L^{2}(\Omega)}+\|\nabla \mu_0\|_{L^{2}(\Omega)}+\rho_0 \leqslant \epsilon_0,
	\end{equation}
	Then, the initial-boundary problem \eqref{1.1}-\eqref{1.4} with initial value $(\rho_0,u_0,\phi_0,\mu_0)$, admits a unique global solution $(\rho,u,\phi,\mu)$ on $[0,T]$ satisfying \eqref{1.10}.	Moreover, the following time decay
	\begin{equation}\label{1.13}
		\int_{\Omega}\big(\frac{\rho u^{2}}{2}+\frac{\rho(\phi^{2}-1)^{2}}{4}+\frac{|\nabla \phi|^{2}}{2}\big) dx
		\leqslant C\epsilon_0 e^{-a_0t}+\frac{a_0}{4}\int_\Omega\rho_0dx
	\end{equation}
	holds for any $t\in [0,T]$, where  $a_0=\left(\max\{\frac{\sqrt{2}}{2v_*}, C_0\epsilon_0\}\epsilon_0\right)^{-1}$ and $C_0$ depends only on $|\Omega|$, $v_*$, $v^*$, $\|\mu_0\|_{L^{2}(\Omega)}$, $\|\phi_0\|_{L^{2}(\Omega)}.$
\end{theorem}

\begin{remark}
	Compared with the previous known results on the global existence of strong solution \cite{2012-AH,2019-CL,Cao-2012,Chen-He-Mei-Shi-2018,Gal-2010,Gui-2018,Qiu-Tang-Wang-2023}, our Theorem \ref{thm1.2} successfully proves for the first time that, given sufficiently small initial conditions $\eqref{1.12}$, the local strong solution to the system $\eqref{1.1}$ in three dimensions can be extended to a global strong solution. Furthermore, it provides insights into the decay behavior of the solution to system $\eqref{1.1}$.
\end{remark}
\begin{remark}
	To the best of our knowledge, the decay-in-time property \eqref{1.13} in Theorem \ref{thm1.2} is proven under the condition of sufficiently small initial data $\eqref{1.12}$ and the choice of Landau potential $\eqref{1.2}$. Additionally, using the energy inequality, it is shown that the decay of the solution is exponential, and it also depends on the initial density, which is new and somewhat surprising, since the known corresponding decay-in-time rates for the strong solution to the incompressible Navier-Stokes-Cahn-Hilliard system are algebraic (see \cite{Gal-2016}).
\end{remark}
The structure of this paper is as follows. We summarize preliminary results in Section 2. Then we establish the blow-up criterion for the local  strong solutions to the initial-boundary problem \eqref{1.1}-\eqref{1.4} in Section 3. Section 4 is devoted to the Proof of Theorem \ref{thm1.2}.

\section{Preliminaries}

At first, we review some elementary inequalities and important lemmas that are used extensively in this paper. The well-known Gagliardo-Nirenberg inequality be used frequently in the later analysis \ (see $\cite{1968Linear}$).
\begin{lemma} \label{lem2-1}
	{\upshape(\textbf{Gagliardo-Nirenberg Inequality}).} Let $\Omega$ be a domain of $\mathbb{R}^{3}$ with smooth boundary $\partial\Omega$.
	Let the positive constants $j$ and $k$  satisfy $0\leqslant  j<k.$ 
	For $1\leqslant  p,r\leqslant\infty$, there exists some generic constants $C>0$ that may depend on $p$ and $r$ such that for any $u\in W_0^{k,p}(\Omega),$  we have
	$$\|D^{j}u\|_{L^{q}(\Omega)}\leqslant  C\|D^{k}u\|^{\theta}_{L^{p}(\Omega)}\|u\|^{1-\theta}_{L^{r}(\Omega)}$$
	where $\frac{j}{k}\leqslant  \theta\leqslant 1$ and $\frac{1}{q}=\frac{j}{n}+\theta(\frac{1}{p}-\frac{k}{n})+\frac{1-\theta}{r}.$
\end{lemma}

Next, we introduce the following auxiliary result, which is given in \cite{2004Dynamics}.

\begin{lemma}\label{lem2-2}
	Let $v\in W^{1,2}(\Omega)$, and let $\rho$ be a non-negative function such that
	$$0<M\leqslant\int_{\Omega}\rho   dx, \ \ \int_{\Omega}\rho^{\gamma} dx\leqslant  E_0,$$
	where $\Omega\subset \mathbb{R}^{3}$ is a bounded domain, $M, E_0$ are positive constants, and $\gamma>1$.
	
	Then there exists a constant $c$ depending solely on $M, E_0$ such that
	$$\|v\|_{L^{2}(\Omega)}^{2}\leqslant  c(E_0,M)\big(\|\nabla_x v\|_{L^{2}(\Omega)}^{2}+\big(\int_{\Omega}\rho|v | dx\big)^{2}\big).$$
\end{lemma}

For convenience's sake, we state a standard elliptic theory here, which can be found in \cite{Abels-2009}.

\begin{lemma}\label{lem2-3}
	For $1\leqslant  q\leqslant  \infty,$ we set
	$$L^{q}_{(m)}(\Omega)=\left\{f\in L^{q}(\Omega):m(f)=m\right\}\mbox{ with } m(f)=\frac{1}{|\Omega|}\int_{\Omega}f(x) dx. $$
	Let $u\in W^{2,q}(\Omega), \ \Delta u=f$ almost everywhere in a suitable bounded
	domain $\Omega\subset {\mathbb R}^3$, and $\partial_n u|_{\partial\Omega}=0$ in the sense of trace. If $f\in W^{1,q}(\Omega)\cap L^{q}_{(0)}(\Omega)$ and $\partial\Omega\in C^3,$ then $u\in W^{3,q}(\Omega).$ Moreover,
	$$\|u\|_{W^{k+2,q}(\Omega)}\leqslant  C_q\|f\|_{W^{k,q}(\Omega)}$$
	for all any $f\in W^{k,q}(\Omega)\cap L^{q}_{(0)}(\Omega) \ ( k=0,1),$ where the positive constant $C_q$ depends only on $1<q<\infty$ and $\Omega.$
\end{lemma}

Now, we introduce a higher regularity for the stationary solution to the following Stokes system
\begin{equation*}
	\begin{cases}
		\partial_t u-\operatorname{div}(v(\phi)\mathbb{D}u)+\nabla p=f\ \ \ \ \mbox{in}\ \Omega\times(0,T),\\
		\operatorname{div}u=0\ \ \ \ \mbox{in}\ \Omega\times(0,T),\\
		u|_{\partial\Omega}=0\ \ \ \ \mbox{on}\ \partial\Omega\times(0,T),\\
		u|_{t=0}=u_0\ \ \ \ \mbox{in}\ \Omega.
	\end{cases}
\end{equation*}
This higher regularity for the stationary solution can be found in \cite{Abels-2009}.

\begin{lemma}\label{lem2-4}
	Let $\Omega$ be a suitable bounded domain in ${\mathbb R}^3$ with $\partial\Omega\in C^3,$  $v\in \mathnormal{C}^{2}(\mathbb{R}), \phi\in W^{1+j,r}(\Omega)\ (j=0,1)$, $r>3$ and $\|\phi\|_{W^{1+j,r}(\Omega)}\leqslant  R.$ Assume that $u\in V_2^{1+j}(\Omega)$ is the solution to the following equation
	$$(v(\phi)Du,D\psi)_{L^{2}(\Omega)}=(f,\psi)_{L^{2}(\Omega)} \ \ \ (\forall \psi\in \mathnormal{C}_{0,\sigma}^{\infty}(\Omega)),$$
	where $f\in H^{s}(\Omega)$ and $s\in[0,j].$ Then $u\in H^{2+s}(\Omega)$ with the property
	$$\|u\|_{H^{2+s}(\Omega)}\leqslant  C(R)\|f\|_{H^{s}(\Omega)},$$
	where $C(R)$ only depends on $\Omega,v,r$, $R>0$, and the space $V_2^{1+j}(\Omega)$ is defined as
	$$V_2^{1+j}(\Omega)=H^{1+j}(\Omega)\cap H^{1}_{0}(\Omega)\cap H^{1}(\Omega)\cap L^{2}_{\sigma}(\Omega) \ \ \ (j=0,1).$$
\end{lemma}

The following well-known Gronwall's Lemma will be frequently used later.

\begin{lemma}(Grownwall inequality \cite{2010PDE})\label{lem2-5}
	
	(i) \ Let $\eta(\cdot)$ be a nonnegative, absolutely continuous function on $[0,T]$ satisfying
	$$\eta^{\prime}(t)\leqslant\phi(t)\eta(t)+\psi(t),$$ where $\phi(t)$ and $\psi(t)$ are nonnegative, summable functions on $[0,T]$. Then
	$$\eta(t)\leqslant exp\big({\int_{0}^{t}\phi(s)ds}\big)\big(\eta(0)+\int_{0}^{t}\psi(s)ds\big),$$
	for all $0\leqslant t \leqslant T$.
	
	(ii) \ In particular, if
	$$\eta^{\prime}\leqslant \phi\eta \ \ \text{on} \ \ [0,T] \ \ \text{and} \ \ \eta(0)=0,$$
	then
	$$\eta \equiv 0 \ \ \text{on} \ \ [0,T].$$
\end{lemma}

In the end, we state the existence of local strong solution to the initial-boundary problem \eqref{1.1}-\eqref{1.4}, which is obtained by  Giorgini and Temam in \cite{{Giorgini-2020}}.

\begin{lemma}\label{Existence}
	Let $\Omega$ be a bounded domain of class $\mathnormal{C}^3$ in $\mathbb{R}^{3}$. Assume that $\rho_0\in L^{\infty}(\Omega)$, $u_0 \in V_{\sigma}(\Omega)$ and $\phi_0 \in H^{2}(\Omega)$ are given such that $0 <\rho_*\leqslant \rho_0 \leqslant\rho^{*}$, $\partial_n\phi_0 = 0$ on $\partial\Omega$, and $\mu_0 = -\frac{\Delta\phi_0}{\rho_0}+\Psi^{\prime}(\phi_0)\in H^{1}(\Omega)$. Then, there exist $T^{*} > 0$, depending on the norms of the initial data, and a strong solution $(\rho,u,\phi,\mu)$ to the initial-boundary problem \eqref{1.1}-\eqref{1.4} on $[0,T^*]$ satisfying
	\begin{equation}\label{existence}
		\begin{cases}
			\rho\in \mathnormal{C}([0,T^{*}];L^{r}(\Omega))\cap L^{\infty}(\Omega\times (0,T^{*}))\cap L^{\infty}(0,T^{*};H^{-1}(\Omega)),\\
			u\in \mathnormal{C}([0,T^{*}];V_\sigma(\Omega))\cap L^{2}(0,T^{*};H^{2}(\Omega))\cap H^{1}(0,T^{*};H_\sigma(\Omega)),\\
			\phi\in \mathnormal{C}([0,T^{*}];(W^{2,6}(\Omega))_{w})\cap H^{1}(0,T^{*};H^{1}(\Omega)),\\
			\mu\in L^{\infty}(0,T^{*};H^{1}(\Omega))\cap L^{2}(0,T^{*};W^{2,6}(\Omega))
		\end{cases}
	\end{equation}
	for any $r\in[2,\infty]$.
\end{lemma}

\section{Blow-up criterion}\label{S3}

Throughout the rest of this subsection, we denote by $C$ a generic constant depending only on depending only on $\Omega, \ v_*, \ v^*,$ $\|\sqrt{\rho_0}u_0\|_{L^2(\Omega)},$  $\|\nabla u_0\|_{L^2(\Omega)},$  $\|\nabla \mu_0\|_{L^2(\Omega)}, \ \|\nabla \phi_0\|_{L^2(\Omega)}.$  The positive constant $C(\rho^*)$ depending on $\Omega$,  $v_*$, $v^*$, $\rho^*$ and $ \|\nabla u_0\|_{L^2(\Omega)},$ $\|\nabla \mu_0\|_{L^2(\Omega)},$ $\|\nabla \phi_0\|_{L^2(\Omega)}.$ Here, we make the positive constant $C(\rho^*)$ special in order to light on our next section.

Let $T^*\in (0,\infty)$ be the maximum time of existence of strong solution $(\rho,u,\phi,\mu)$ to the initial-boundary problem \eqref{1.1}-\eqref{1.4}. Namely, $(\rho,u,\phi,\mu)$ is a strong solution to \eqref{1.1}-\eqref{1.4} in $\Omega\times[0,T]$ for any $T\in (0,T^*),$ but not a strong solution in $\Omega\times[0,T^*].$ Suppose that \eqref{1.11} were false, i.e.
\begin{equation}\label{3.1}
	M:=\|u\|_{L^{\frac{4r}{r-6}}(0,T^*;L^{r}(\Omega))}<\infty.
\end{equation}
The goal is to show that under the assumption \eqref{3.1}, there is a bound $C>0$ such that
\begin{equation}\label{3.2}
	0<\frac{\rho_*}{C}\leqslant \varrho \leqslant C\rho^*,
\end{equation}
and there is a bound $C(\rho^*)>0$ such that
\begin{align}\label{3.3}
		\sup_{t \in [0,T]}\big(\|\nabla u\|_{L^{2}(\Omega)}^{2}+\|\nabla \mu\|_{L^{2}(\Omega)}^{2}+\|\phi\|_{W^{2,6}(\Omega)}^{2} \big)+ \int_{0}^{T}(\| \phi_t\|_{H^{1}(\Omega)}^{2}+\| u_t \|_{L^{2}(\Omega)}^{2}+\|u\|_{H^{2}(\Omega)}^{2}) dt\leqslant  C(\rho^*).
\end{align}

Obviously, the lower bound and the upper bound of the density is obtained easily.

\begin{lemma}\label{lem3-1}
	Under the conditions of Theorem \ref{thm 1.1}, there exists a positive constant $C>0$ such that
	$$0<\frac{\rho_*}{C}\leqslant\rho(x,t)\leqslant C\rho^*$$
	holds for any $(x,t)\in \Omega \times [0,T]$ with $T\in (0,T^{*}).$
\end{lemma}

The basic energy inequality of the strong solution to the problem \eqref{1.1}-\eqref{1.4} is stated in the following lemma, whose proof can be found in \cite{Giorgini-2020}. The basic energy inequality plays a crucial role in our analysis, and its proof is given here for the convenience of readers.

\begin{lemma}\label{lem3-2}
	Under the conditions of Theorem \ref{thm 1.1}, the energy inequality
	\begin{align}\label{3.4}
			\sup_{t\in [0,T]}\int_{\Omega}(\frac{\rho u^{2}}{2}+\frac{\rho(\phi^{2}-1)^{2}}{4}+\frac{|\nabla \phi|^{2}}{2}) dx +\int_{0}^{T}\int_{\Omega}(v(\phi)|{\mathbb D} u|^{2}+|\nabla \mu |^{2}) dx dt\leqslant  C
	\end{align}
	holds for any $T\in (0,T^{*}).$
\end{lemma}

\begin{proof}
	Multiplying $\eqref{1.1}_2$ by $u$ and integrating over $\Omega$, one obtains that
	\begin{equation}\label{3.5}
		\frac{d}{dt}\int_{\Omega}\frac{1}{2}\rho|u |^{2} dx
		+\int_{\Omega}v(\phi)|\mathbb{D}u|^{2} dx=-\int_{\Omega}\operatorname{div}(\nabla\phi\otimes\nabla\phi)\cdot u dx.
	\end{equation}
	Next, one multiplies $\eqref{1.1}_4$ by $\mu$ and integrates the resultant over $\Omega$ to obtain that
	\begin{align}\label{3.6}
			\frac{d}{dt}\int_{\Omega}\big(\frac{1}{2}|\nabla\phi|^{2}+\rho\Psi(\phi)\big) dx+\int_{\Omega}\rho u \cdot\nabla\phi\mu dx+\int_{\Omega}|\nabla\mu|^{2} dx=\int_{\Omega}\rho_t\Psi(\phi)dx
	\end{align}
	by using $\eqref{1.1}_5$ and  $\eqref{1.2}.$ Recalling that
	\begin{align*}
		-\operatorname{div}(\nabla\phi\otimes\nabla\phi)&=-\Delta \phi\nabla\phi-\nabla(\frac{1}{2}|\nabla\phi|^{2})\\
		&=\rho\mu\nabla\phi-\rho\Psi^{\prime}(\phi)\nabla\phi-\nabla(\frac{1}{2}|\nabla\phi|^{2})\\
		&=\rho\mu\nabla\phi-\rho\nabla\Psi(\phi)-\nabla(\frac{1}{2}|\nabla\phi|^{2}),
	\end{align*}
	one deduces from \eqref{3.6} that
	\begin{equation}\label{3.7}
		\frac{d}{dt}\int_{\Omega}\frac{1}{2}\rho|u |^{2} dx+\int_{\Omega}v(\phi)|\mathbb{D}u|^{2} dx=\int_{\Omega}\rho\mu\nabla\phi\cdot u dx-\int_{\Omega}\rho u \cdot \nabla\Psi(\phi) dx.
	\end{equation}
	Summing $\eqref{3.5}$ and $\eqref{3.7}$, one arrives at
	$$\frac{d}{dt}\int_{\Omega}\big(\frac{1}{2}\rho|u |^{2}+\frac{1}{2}|\nabla\phi|^{2}+\rho\Psi(\phi)\big) dx
	+\int_{\Omega}v(\phi)|\mathbb{D}u|^{2} dx+\int_{\Omega}|\nabla\mu|^{2} dx=0.$$
	Integrating the above equation in time, one gets that
	\begin{align}\label{3.8}
		&\int_{\Omega}\big(\frac{1}{2}\rho|u|^{2}+\frac{1}{2}|\nabla\phi|^{2}+\rho\Psi(\phi)\big) dx
		+\int_{0}^{t}\int_{\Omega}\big(v(\phi)| \mathbb{D}u|^{2}+|\nabla\mu|^{2}\big) dxd\tau\nonumber \\
		&=\int_{\Omega}\frac{1}{2}\rho_0 |u_0 |^{2} + \frac{1}{2}|\nabla\phi_0 |^{2}+\rho_0 \Psi(\phi_0) dx
	\end{align}
	for all $t\in [0,T]$. The desired inequality \eqref{3.4} is obtained.
\end{proof}

The key estimates on $\|\phi\|_{H^{1}(\Omega)},  \|\mu\|_{H^{1}(\Omega)}$ and $\|u\|_{H^{2}(\Omega)}$ will be given in the following lemma.

\begin{lemma}\label{lem3-3}
	Under the conditions of Theorem \ref{1.1} and \eqref{3.1}, we can obtain
	\begin{eqnarray}
		&&\sup_{t\in [0,T]}\|\phi\|_{H^{1}(\Omega)}\leqslant  C,\label{3.9}\\
		&&\|\mu\|_{H^{1}(\Omega)}\leqslant  C (1+\|\nabla \mu\|_{L^{2}(\Omega)}),\label{3.10}\\
		&&\|\phi\|_{W^{2,6}(\Omega)}\leqslant  C (1+\|\nabla\mu\|_{L^{2}(\Omega)}),\label{3.11} \\
		&&\|u\|_{H^{2}(\Omega)}\leqslant  C (1+\|\nabla\mu\|_{L^{2}(\Omega)}^{\frac{5}{4}}+\| \sqrt{\rho}u_t \|_{L^{2}(\Omega)}+\|u\|_{L^{6}(\Omega)}^{3}),\label{3.12}
	\end{eqnarray}
	for any $T\in (0,T^{*}).$
\end{lemma}

\begin{proof}
	
	Firstly, it is deduced from Lemma \ref{lem3-2} that
	$$\int_{\Omega}(\rho\phi^{4}+\rho) dx\leqslant C+2\int_{\Omega}\rho\phi^{2} dx\leqslant C+\int_{\Omega}\frac{1}{2}\rho\phi^{4} dx,$$
	and so
	$$\int_{\Omega}\rho\phi^{4} dx\leqslant C.$$
	Thus, one gets that
	\begin{equation}\label{3.13}\int_{\Omega}\rho|\phi \rvert  dx\leqslant \big(\int_{\Omega}\rho|\phi |^{4} dx\big)^{\frac{1}{4}}\big(\int_{\Omega}\rho dx\big)^{\frac{3}{4}}\leqslant  C.
	\end{equation}
	Combining \eqref{3.4} and \eqref{3.7}, one obtains from Lemma \ref{lem2-2} that
	$$\sup_{t\in [0,T]}\|\phi\|_{H^{1}(\Omega)}\leqslant  C$$
	for any $T\in (0,T^*).$	
	
	Secondly, multiplying $ \mu $ on both sides of equation $\eqref{1.1}_5$ and integrating the result equation over $\Omega$, one has that
	\begin{align}\label{3.14}
		\begin{split}
			\int_{\Omega}\rho\mu^{2}dx &=\int_{\Omega}-\Delta \phi\mu dx+\int_{\Omega}\rho\mu (\phi^{2}-1)\phi dx\\
			&=\int_{\Omega}\nabla\phi\cdot\nabla\mu dx + \int_{\Omega}\rho\mu(\phi^{2}-1)\phi dx\\
			&\leqslant\frac{1}{2}\|\nabla\phi\|^{2}_{L^{2}(\Omega)}+\frac{1}{2}\|\nabla\mu\|^{2}_{L^{2}(\Omega)}+\frac{1}{2}\int_{\Omega}\rho\mu^{2} dx+\frac{1}{2}\|\sqrt{\rho}\phi^{2}\|_{L^{2}(\Omega)}\|\sqrt{\rho}(\phi^{2}-1)\|_{L^{2}(\Omega)}\\
			&\leqslant  C(1+\|\nabla\mu\|_{L^{2}(\Omega)}^{2})+\frac{1}{2}\int_{\Omega}\rho\mu^{2} dx,
		\end{split}
	\end{align}
	which implies
	$$\int_{\Omega}\rho\mu^{2} dx\leqslant C(1+\|\nabla \mu\|_{L^{2}(\Omega)}^{2}).$$
	Moreover,
	$$\int_{\Omega}\rho|\mu | dx\leqslant	\big(\int_{\Omega}\rho dx\big)^{\frac{1}{2}}\big(\int_{\Omega}\rho\mu^{2} dx\big)^{\frac{1}{2}}\leqslant C(1+\|\nabla\mu\|_{L^{2}(\Omega)}).$$
	Using Lemma \ref{lem2-2} again, one gets that
	$$\|\mu\|_{H^{1}(\Omega)}\leqslant  C(1+\|\nabla \mu\|_{L^{2}(\Omega)}).$$
	
	Thirdly, we recall that $\Psi(\phi)=\frac{1}{4}(\phi^2-1)^2$ and then $\Psi^{\prime}(\phi)=(\phi^2-1)\phi.$ Now, one applies Lemma \ref{lem2-3} to $\eqref{1.1}_5$ and then obtains from \eqref{3.3} that
	\begin{align*}
		\|\phi\|_{W^{2,6}(\Omega)}&\leqslant  C(\Omega)\|\rho\mu+\rho\Psi^{\prime}(\phi)\|_{L^{6}(\Omega)}\\
		&\leqslant  C(\Omega)(\|\rho\mu\|_{L^{6}(\Omega)}+\|\rho\Psi^{\prime}(\phi)\|_{L^{6}(\Omega)})\\
		&\leqslant  C (1+\|\nabla \mu\|_{L^{2}(\Omega)}+\|\phi\|^{\frac{1}{3}}_{L^{18}(\Omega)}+\|\phi\|_{L^{6}(\Omega)})\\
		&\leqslant  C (1+\|\nabla \mu\|_{L^{2}(\Omega)})+\eta\|\phi\|_{W^{2,6}(\Omega)}
	\end{align*}
	holds for any fixed $\eta \in (0,1).$ Thus
	$$\|\phi\|_{W^{2,6}(\Omega)}\leqslant  C(1+\|\nabla \mu\|_{L^{2}(\Omega)}).$$

	Finally,  one can apply Lemma \ref{lem2-4} on $\eqref{1.1}_2$ to arrive that
	\begin{align*}
		\|u\|_{H^{2}(\Omega)}
		&\leqslant C(\Omega) \|-\Delta\phi\nabla\phi-\nabla^{2}\phi\cdot\nabla\phi-\rho u_t -\rho(u\cdot\nabla)u\|_{L^{2}(\Omega)}\\
		&\leqslant C(\Omega) (\|\Delta\phi\|_{L^{6}(\Omega)}\|\nabla\phi\|_{L^{3}(\Omega)}
		+\|\nabla^{2}\phi\|_{L^{6}(\Omega)}\|\nabla\phi\|_{L^{3}(\Omega)}+\|\sqrt{\rho} u_t \|_{L^{2}(\Omega)}\\
		&\quad+\|\nabla u\|_{L^{3}(\Omega)}\|u\|_{L^{6}(\Omega)})\\
		&\leqslant C(\Omega) (\|\phi\|_{W^{2,6}(\Omega)}^{\frac{5}{4}}\|\nabla\phi\|_{L^{2}(\Omega)}^{\frac{3}{4}}
		+\|\sqrt{\rho} u_t \|_{L^{2}(\Omega)}+\| u\|_{H^{2}(\Omega)}^{\frac{1}{2}}\|u\|_{L^{6}(\Omega)}^{\frac{3}{2}})\\
		&\leqslant  C(1+\|\nabla\mu\|_{L^{2}(\Omega)}^{\frac{5}{4}}+\|\sqrt{\rho} u_t \|_{L^{2}(\Omega)}
		+\|u\|_{H^{2}(\Omega)}^{\frac{1}{2}}\|u\|_{L^{6}(\Omega)}^{\frac{3}{2}})\\
		&\leqslant  C(1+\|\nabla\mu\|_{L^{2}(\Omega)}^{\frac{5}{4}}
		+\|\sqrt{\rho} u_t \|_{L^{2}(\Omega)}+\|u\|_{L^{6}(\Omega)}^{3})+\eta\|u\|_{H^{2}(\Omega)}
	\end{align*}
	holds for any fixed $\eta \in (0,1).$ Thus, one gets that
	$$\|u\|_{H^{2}(\Omega)}\leqslant  C(1+\|\nabla\mu\|_{L^{2}(\Omega)}^{\frac{5}{4}}+\| \sqrt{\rho}u_t \|_{L^{2}(\Omega)}+\|u\|_{L^{6}(\Omega)}^{3}).$$
	The proof of Lemma \ref{lem3-3} is completed.
\end{proof}

With the help of Lemma \ref{lem3-2} and Lemma \ref{lem3-3}, we are able to obtain the following lemma.

\begin{lemma}\label{lem3-4}
	Under the conditions of Theorem \ref{1.1} and \eqref{3.1}, it holds that
	\begin{align*}
\sup_{t\in [0,T]}\big(\|\nabla u\|^{2}_{L^{2}(\Omega)}+\|\nabla \mu\|^{2}_{L^{2}(\Omega)}+\|\phi\|^{2}_{W^{2,6}(\Omega)} \big)+\int_{0}^{T}\big(\| \phi_t\|_{H^{1}(\Omega)}^{2}+\|  u_t \|_{L^{2}(\Omega)}^{2}+\|u\|_{H^{2}(\Omega)}^{2}\big) dt\leqslant  C
	\end{align*}
	for any $T\in (0,T^{*}).$
\end{lemma}

\begin{proof}
	Multiplying $ u_t $ on both sides of equation $\eqref{1.1}_2$ and integrating the result equation over $\Omega$, one has that
	\begin{align}\label{3.15}
		\begin{split}
			&\frac{1}{2}\frac{d}{dt}\int_{\Omega}v(\phi)|\mathbb{D}u |^{2} dx+\int_{\Omega}\rho|u_t  |^{2} dx+\int_{\Omega}\rho(u\cdot \nabla)u\cdot u_t dx\\
			&=\int_{\Omega}\rho\mu\nabla\phi\cdot u_t  dx-\int_{\Omega}\rho\nabla\Psi(\phi)\cdot u_t  dx+\int_{\Omega}v^{\prime}(\phi) \phi_t |\mathbb{D}u |^{2} dx.
		\end{split}
	\end{align}
	Next, multiplying the equation $\eqref{1.1}_4$ by $ \mu_t $ and integrating over $\Omega$, one finds that
	\begin{equation}\label{3.16}
		\frac{1}{2}\frac{d}{dt}\|\nabla\mu\|_{L^{2}(\Omega)}^{2}+\int_\Omega(\rho \phi_t\mu_t+\rho u\cdot\nabla\phi\mu_t) dx=0.
	\end{equation}
	Using $\eqref{1.1}_1$ and $\eqref{1.1}_5$, one obtains from \eqref{3.16} that
	\begin{align}\label{3.17}
		\begin{split}
			&\frac{d}{dt}\left\{\frac{1}{2}\|\nabla\mu\|_{L^{2}(\Omega)}^{2}+\int_{\Omega}\rho\mu u\cdot\nabla\phi  dx\right\}+\int_{\Omega}|\nabla\phi_t  |^{2} dx+3\int_{\Omega}\rho\phi^{2}\phi^{2}_{t}dx-\int_{\Omega}\rho\phi^{2}_{t} dx\\
			&=-\int_{\Omega}\rho\Psi^{\prime\prime} (\phi) \phi_t  u\cdot\nabla\phi  dx-\int_{\Omega}\rho\Psi^{\prime}(\phi)u\cdot\nabla\phi_t   dx+\int_{\Omega}\rho u_t \cdot\nabla\mu dx\\
			&\quad+\int_{\Omega}\rho(u\cdot\nabla\mu)(u\cdot\nabla\phi) dx+\int_{\Omega}\rho\mu u\cdot\nabla(u\cdot\nabla\phi) dx+\int_{\Omega}\rho\mu \phi_t dx+2\int_{\Omega}\rho\mu u\cdot\nabla\phi_t dx.
		\end{split}
	\end{align}
	Now, one multiplies the equation $\eqref{1.1}_4$ by $ \phi_t $ and integrates the resultant over $\Omega$ to get that
	$$\int_{\Omega}\rho| \phi_t  |^{2} dx=-\int_{\Omega}\rho u\cdot\nabla\phi \phi_t   dx-\int_{\Omega}\nabla\mu\cdot\nabla\phi_t dx.$$
	Since $\Psi^{\prime\prime} (s)=3s^{2}-1\geqslant -1$, one can multiply the above equation by a factor 2 and add the resultant with \eqref{3.17} to arrive at
	\begin{align}\label{3.18}
		\begin{split}
			&\frac{d}{dt}\left\{\frac{1}{2}\|\nabla\mu\|_{L^{2}(\Omega)}^{2}+\int_{\Omega}\rho\mu u\cdot\nabla\phi  dx\right\}+\int_{\Omega}|\nabla\phi_t  |^{2} dx+\int_{\Omega}\rho| \phi_t  |^{2} dx+3\int_{\Omega}\rho\phi^{2}|\phi_t \rvert^{2} dx\\
			&=-\int_{\Omega}\rho\Psi^{\prime\prime} (\phi) \phi_t  u\cdot\nabla\phi  dx-\int_{\Omega}\rho\Psi^{\prime}(\phi)u\cdot\nabla\phi_t dx+\int_{\Omega}\rho u_t \cdot\nabla\mu dx\\
			&\quad+\int_{\Omega}\rho(u\cdot\nabla\mu)(u\cdot\nabla\phi) dx+\int_{\Omega}\rho\mu u\cdot\nabla(u\cdot\nabla\phi) dx+\int_{\Omega}\rho\mu \phi_t dx\\
			&\quad+2\int_{\Omega}\rho\mu u\cdot\nabla\phi_t   dx-2\int_{\rho}u\cdot\nabla\phi \phi_t   dx-2\int_{\Omega}\nabla\mu\cdot\nabla\phi_t dx.
		\end{split}
	\end{align}
	Summing \eqref{3.15} and \eqref{3.18}, one arrives at
	\begin{align}\label{3.19}
		&\frac{d}{dt}\left\{\int_{\Omega}\frac{v(\phi)}{2}|\mathbb{D}u |^{2} dx+\frac{1}{2}\|\nabla\mu\|_{L^{2}(\Omega)}^{2}
		+\int_{\Omega}\rho\mu u\cdot\nabla\phi  dx\right\}\nonumber\\
		&\quad+\int_{\Omega}|\nabla\phi_t  |^{2} dx+\int_{\Omega}\rho| \phi_t  |^{2} dx+\int_{\Omega}\rho| u_t  |^{2} dx+3\int_{\Omega}\rho\phi^{2}| \phi_t \rvert^{2} dx\nonumber\\
		&=\underbrace{-\int_{\Omega}\rho(u\cdot\nabla)u \cdot u_t   dx}_{I_1}+\underbrace{2\int_{\Omega}\rho\mu\nabla\phi\cdot u_t dx}_{I_2}+\underbrace{\int_{\Omega}-\rho\nabla\Psi(\phi)\cdot u_t dx}_{I_3}\nonumber\\
		&\quad+\underbrace{\int_{\Omega}v^{\prime}(\phi) \phi_t |\mathbb{D}u |^{2} dx}_{I_4}+\underbrace{\int_{\Omega}-\rho\Psi^{\prime\prime} (\phi) \phi_t  u\cdot\nabla\phi dx}_{I_5}+\underbrace{\int_{\Omega}-\rho\Psi^{\prime}(\phi)u\cdot\nabla\phi_t dx}_{I_6}\nonumber\\
		&\quad+\underbrace{\int_{\Omega}\rho \phi_t  u\cdot\nabla\mu dx}_{I_7}+\underbrace{\int_{\Omega}\rho(u\cdot\nabla\mu)(u\cdot\nabla\phi) dx}_{I_8}+\underbrace{\int_{\Omega}\rho\mu u\cdot\nabla(u\cdot\nabla\phi) dx}_{I_9}\nonumber\\
		&\quad+\underbrace{2\int_{\Omega}\rho\mu u\cdot\nabla\phi_t   dx}_{I_{10}}+\underbrace{2\int_{\Omega}-\rho u\cdot\nabla\phi \phi_t  dx}_{I_{11}}+\underbrace{2\int_{\Omega}-\nabla\mu\cdot\nabla\phi_t dx}_{I_{12}}.
	\end{align}
	With the help of Lemma \ref{lem3-3}, Sobolev inequality, H$\ddot{\text{o}}$lder inequality and Young inequality, each term on the right hand of \eqref{3.19} is estimated as follows. That is,
	\begin{align*}
		|I_1|
		&=|\int_{\Omega}\rho (u\cdot \nabla)u\cdot u_t  dx |\\
		&\leqslant  C(\rho^{*})\|\sqrt{\rho} u_t \|_{L^{2}(\Omega)}\|u\|_{L^{6}(\Omega)}\|\nabla u\|_{L^{3}(\Omega)}\\
		&\leqslant C(\rho^{*})\|\sqrt{\rho}u_t\|_{L^{2}(\Omega)}\|u\|_{L^{6}(\Omega)}^{\frac{3}{2}}\|u\|_{H^{2}}^{\frac{1}{2}} \\
		&\leqslant\eta\|\sqrt{\rho} u_t \|_{L^{2}(\Omega)}^{2}+C(\rho^{*},\eta)\|u\|_{L^{6}(\Omega)}^{3}\| u\|_{H^{2}(\Omega)}\\
		&\leqslant\eta\|\sqrt{\rho} u_t \|_{L^{2}(\Omega)}^{2}+C(\rho^{*},\eta)(1+\|u\|_{L^{6}(\Omega)}^{3}+\|\nabla \mu\|_{L^{2}(\Omega)}^{\frac{5}{4}}+\|\sqrt{\rho} u_t \|_{L^{2}(\Omega)})\|u\|_{L^{6}(\Omega)}^{3}\\
		&\leqslant\eta\|\sqrt{\rho} u_t \|_{L^{2}(\Omega)}^{2}+C(\rho^{*},\eta)\|\nabla u\|_{L^{2}(\Omega)}^{2}(1+\|u\|_{L^{6}(\Omega)}^{4}+\|\nabla\mu\|_{L^{2}(\Omega)}^{2}),\\
		|I_2|&=2|\int_{\Omega}\rho\mu\nabla\phi\cdot u_t dx| \\
		&\leqslant  C(\rho^{*})\|\sqrt{\rho} u_t \|_{L^{2}(\Omega)}\|\mu\|_{L^{6}(\Omega)}\|\nabla\phi\|_{L^{3}(\Omega)}\\
		&\leqslant  C(\rho^{*})\|\sqrt{\rho} u_t \|_{L^{2}(\Omega)}\|\mu\|_{H^{1}(\Omega)}\|\nabla\phi\|_{L^{2}(\Omega)}^{\frac{3}{4}}\|\phi\|_{W^{2,6}(\Omega)}^{\frac{1}{4}}\\
		&\leqslant\eta\|\sqrt{\rho} u_t \|_{L^{2}(\Omega)}^{2}+C(\rho^{*},\eta)(1+\|\nabla\mu\|_{L^{2}(\Omega)}^{\frac{5}{2}}),\\
		|I_3|&=|\int_{\Omega}\rho\nabla\Psi(\phi)\cdot u_t dx | \\
		&\leqslant  C(\rho^{*})\|\Psi^{\prime}(\phi)\|_{L^{\infty}(\Omega)}\|\nabla\phi\|_{L^{2}(\Omega)}\|\sqrt{\rho} u_t \|_{L^{2}(\Omega)}^{2}\\
		&\leqslant  C(\rho^*)\|\phi^{3}-\phi\|_{L^{\infty}(\Omega)}\|\sqrt{\rho} u_t \|_{L^{2}(\Omega)}^{2}\\
		&\leqslant\eta\|\sqrt{\rho} u_t \|_{L^{2}(\Omega)}^{2}+C(\rho^{*},\eta)(\|\phi\|_{L^{\infty}(\Omega)}^{6}+\|\phi\|_{L^{\infty}(\Omega)}^{2})\\
		&\leqslant\eta\|\sqrt{\rho} u_t \|_{L^{2}(\Omega)}^{2}+C(\rho^{*},\eta)(1+\|\phi\|_{L^{2}(\Omega)}^{3}\|\phi\|_{W^{2,6}(\Omega)}^{3})\\
		&\leqslant\eta\|\sqrt{\rho} u_t \|_{L^{2}(\Omega)}^{2}+C(\rho^{*},\eta)(1+\|\nabla\mu\|_{L^{2}(\Omega)}^{3}),\\
		|I_4 |&=|\int_{\Omega}v^{\prime}(\phi)\phi_t|\mathbb{D}u |^{2} dx |   \\
		&\leqslant  C\| \phi_t \|_{L^{6}(\Omega)}\|\nabla u\|_{L^{3}(\Omega)}\|\nabla u\|_{L^{2}(\Omega)}\\
		&\leqslant C\|\phi_t\|_{L^{6}(\Omega)}\|\nabla u\|_{L^{2}(\Omega)}\|u\|_{L^{r}(\Omega)}^{\frac{r}{6+r}}\|u\|_{H^{2}(\Omega)}^{\frac{6}{6+r}} \\
		&\leqslant\eta\| \phi_t\|_{L^{6}(\Omega)}^{2}+C(\rho^{*},\eta)\|\nabla u\|_{L^{2}(\Omega)}^{2}\|u\|_{L^{r}(\Omega)}^{\frac{2r}{6+r}}\|u\|_{H^{2}(\Omega)}^{\frac{12}{6+r}}\\
		&\leqslant  \eta\| \phi_t\|_{L^{6}(\Omega)}^{2}+C(\rho^{*},\eta)\|\nabla u\|_{L^{2}(\Omega)}^{2}\|u\|_{L^{r}(\Omega)}^{\frac{2r}{6+r}}(1+\|\nabla \mu\|_{L^{2}(\Omega)}^{\frac{15}{6+r}}\|\sqrt{\rho}u_t\|_{L^{2}(\Omega)}^{\frac{12}{6+r}}+\|u\|_{L^{6}(\Omega)}^{\frac{36}{6+r}})\\
		&\leqslant \eta(\|\phi_t\|_{L^{6}(\Omega)}^{2}+\|\sqrt{\rho}u_t\|_{L^{2}(\Omega)}^{2})+C(\rho^{*},\eta)(\|\nabla u\|_{L^{2}(\Omega)}^{4}+\|u\|_{L^{r}(\Omega)}^{\frac{4r}{r-6}}+\|
		\nabla \mu\|_{L^{2}(\Omega)}^{4}+\|u\|_{L^{r}(\Omega)}^{\frac{36+2r}{6+r}}\|\nabla u\|_{L^{2}(\Omega)}^{2}) \\
		&\leqslant\eta(\|\phi_t\|_{L^{6}(\Omega)}^{2}+\|\sqrt{\rho}u_t\|_{L^{2}(\Omega)}^{2})+C(\rho^{*},\eta)\|\nabla \mu\|_{L^{2}(\Omega)}^{4}+C(\rho^{*},\eta)(\|\nabla u\|_{L^{2}(\Omega)}^{2} +\|u\|_{L^{r}(\Omega)}^{\frac{4r}{r-6}})(1+\|\nabla u\|_{L^{2}(\Omega)}^{2}),\\
		|I_5 |&=|\int_{\Omega}\rho\Psi^{\prime\prime}(\phi)\phi_t u\cdot \nabla\phi dx |  \\
		&\leqslant  C(\rho^{*})\|\nabla\phi\|_{L^{2}(\Omega)}\|u\|_{L^{r}(\Omega)}\| \phi_t \|_{L^{6}(\Omega)}\|\Psi^{\prime\prime} (\phi)\|_{L^{\frac{3r}{r-3}}(\Omega)}\\
		&\leqslant C(\rho^{*})\|u\|_{L^{r}(\Omega)}\|\phi_t\|_{L^{6}(\Omega)}(1+\|\phi\|_{L^{12}(\Omega)}^{2}) \\
		&\leqslant C(\rho^{*})\|u\|_{L^{r}(\Omega)}\|\phi_t\|_{L^{6}(\Omega)}(1+\|\phi\|_{L^{2}(\Omega)}^{\frac{7}{6}}\|\phi\|_{W^{2,6}(\Omega)}^{\frac{5}{6}}) \\
		&\leqslant\eta\| \phi_t \|^{2}_{L^{6}(\Omega)}+C(\rho^{*},\eta)\|u\|^{2}_{L^{r}(\Omega)}(1+\|\nabla\mu\|_{L^{2}(\Omega)}^{\frac{5}{3}}),\\
		|I_6 |& = |\int_{\Omega}\rho\Psi^{\prime}(\phi)u\cdot\nabla\phi_t dx|\\
		&\leqslant  C(\rho^{*})\|\Psi^{\prime}(\phi)\|_{L^{\frac{2r}{r-2}}(\Omega)}\|u\|_{L^{r}(\Omega)}\|\nabla\phi_t \|_{L^{2}(\Omega)}\\
		&\leqslant\eta\|\nabla\phi_t \|_{L^{2}(\Omega)}^{2}+C(\rho^{*},\eta)\|u\|_{L^{r}(\Omega)}^{2}(\|\phi\|_{L^{\infty}(\Omega)}^{3}+1)^{2}\\
		&\leqslant\eta\|\nabla\phi_t \|_{L^{2}(\Omega)}^{2}+C(\rho^{*},\eta)\|u\|_{L^{r}(\Omega)}^{2}(1+\|\phi\|_{L^{2}(\Omega)}^{\frac{3}{2}}\|\phi\|_{W^{2,6}(\Omega)}^{\frac{3}{2}})^{2}\\
		&\leqslant\eta\|\nabla\phi_t \|_{L^{2}(\Omega)}^{2}+C(\rho^{*},\eta)\|u\|_{L^{r}(\Omega)}^{2}(1+\|\nabla\mu\|_{L^{2}(\Omega)}^{3}),\\
		|I_7 |& = |\int_{\Omega}\rho\phi_t u\cdot \nabla\mu dx|\\
		&\leqslant  C(\rho^{*})\| \phi_t \|_{L^{\frac{2r}{r-2}}(\Omega)}\|u\|_{L^{r}(\Omega)}\|\nabla\mu\|_{L^{2}(\Omega)}\\
		&\leqslant\eta\| \phi_t \|_{L^{6}(\Omega)}^{2}+C(\rho^{*},\eta)\|u\|_{L^{r}(\Omega)}^{2}\|\nabla\mu\|_{L^{2}(\Omega)}^{2},\\
		|I_8 |& = |\int_{\Omega}\rho(u\cdot\nabla\mu)(u\cdot\nabla\phi) dx|\\
		&\leqslant  C(\rho^{*})\|u\|_{L^{6}(\Omega)}^{2}\|\nabla\mu\|_{L^{2}(\Omega)}\|\nabla\phi\|_{L^{6}(\Omega)}\\
		&\leqslant  C(\rho^{*})\|\nabla u\|_{L^{2}(\Omega)}^{2}\|\phi\|_{H^{2}(\Omega)}\|\nabla\mu\|_{L^{2}(\Omega)}\\
		&\leqslant  C(\rho^{*})\|\nabla u\|_{L^{2}(\Omega)}^{2}(1+\|\nabla\mu\|_{L^{2}(\Omega)}^{2}),\\
		|I_9 |&= |\int_{\Omega}\rho\mu u\cdot\nabla(u\cdot\nabla\phi) dx|\\
		&\leqslant  C(\rho^{*})\big(\|\mu\|_{L^{6}(\Omega)}\|u\|_{L^{6}(\Omega)}\|\nabla u\|_{L^{2}(\Omega)}\|\nabla\phi\|_{L^{6}(\Omega)}+\|\mu\|_{L^{6}(\Omega)}\|u\|_{L^{6}(\Omega)}\|u\|_{L^{2}(\Omega)}\|\nabla^{2}\phi\|_{{L^6}(\Omega)}\big)\\
		&\leqslant  C(\rho^{*})(1+\|\nabla\mu\|_{L^{2}(\Omega)}^{2})(\|\nabla u\|_{L^{2}(\Omega)}^{2}+1),\\
		|I_{10}|&= |2\int_{\Omega}\rho\mu u\cdot\nabla\phi_t dx|\\
		&\leqslant  C(\rho^{*})\|\mu\|_{L^{\frac{2r}{r-2}}(\Omega)}\|u\|_{L^{r}(\Omega)}\|\nabla\phi_t \|_{L^{2}(\Omega)}\\
		&\leqslant  \eta\|\nabla\phi_t \|_{L^{2}(\Omega)}^{2}+C(\rho^{*},\eta)\|u\|_{L^{r}(\Omega)}^{2}\|\mu\|_{H^{1}(\Omega)}^{2},\\
		&\leqslant \eta\|\nabla\phi_t \|_{L^{2}(\Omega)}^{2}+C(\rho^{*},\eta)\|u\|_{L^{r}(\Omega)}^{2}(1+\|\nabla\mu\|_{L^{2}(\Omega)}^{2}),\\
		|I_{11}|&= |2\int_{\Omega}\rho u\cdot\nabla\phi\phi_t dx|\\
		&\leqslant   C(\rho^{*})\|\sqrt{\rho} \phi_t \|_{L^{2}(\Omega)}\|\nabla\phi\|_{L^{\frac{2r}{r-2}}(\Omega)}\|u\|_{L^{r}(\Omega)}\\
		&\leqslant\eta\|\sqrt{\rho} \phi_t\|_{L^{2}(\Omega)}^{2}+C(\rho^{*},\eta)\|\phi\|_{W^{2,6}(\Omega)}^{2}\|u\|_{L^{r}(\Omega)}^{2}\\
		&\leqslant\eta\|\sqrt{\rho}
		\phi_t\|_{L^{2}(\Omega)}^{2}+C(\rho^{*},\eta)\|u\|_{L^{r}(\Omega)}^{2}(1+\|\nabla\mu\|_{L^{2}(\Omega)}^{2}),\\
		|I_{12} |&= |2\int_{\Omega}\nabla\mu\cdot\nabla\phi_t dx|\\
		&\leqslant  C\|\nabla\mu\|_{L^{2}(\Omega)}\|\nabla\phi_t \|_{L^{2}(\Omega)}\\
		&\leqslant  \eta\|\nabla\phi_t \|_{L^{2}(\Omega)}^{2}+C(\eta)\|\nabla\mu\|_{L^{2}(\Omega)}^{2}
	\end{align*}
	for any $\eta\in (0,1).$ Moreover, it is deduced from Lemma \ref{lem2-2} that
	\begin{equation}\label{3.20}
		\|\phi_t\|_{L^{6}(\Omega)}\leqslant C\big(\|\nabla\phi_t\|_{L^{2}(\Omega)}+(\int_{\Omega}\rho|\phi_t|^{2} dx)^{\frac{1}{2}}\big).
	\end{equation}
	Applying \eqref{3.20} to the estimate for $I_5,$ collecting all the estimates for $I_i$ and taking suitable small $\eta\in(0,1),$ one concludes that
	\begin{align}\label{3.21}
		\begin{split}
			&\frac{d}{dt}\left\{\int_{\Omega}\frac{v(\phi)}{2}|\mathbb{D}u |^{2} dx+\frac{1}{2}\|\nabla\mu\|_{L^{2}(\Omega)}^{2}+\int_{\Omega}\rho\mu u\cdot\nabla\phi  dx\right\}\\
			&\quad+\int_{\Omega}|\nabla\phi_t  |^{2} dx+\int_{\Omega}\rho| \phi_t  |^{2} dx+\int_{\Omega}\rho|u_t  |^{2} dx\\
			&\leqslant  C(\rho^{*})(1+\|u\|_{L^{r}(\Omega)}^{\frac{4r}{r-6}}+\|\nabla\mu\|_{L^{2}(\Omega)}^{2}+\|\nabla u\|_{L^{2}(\Omega)}^{2}) (1+\|\nabla\mu\|_{L^{2}(\Omega)}^{2}+\|\nabla u\|_{L^{2}(\Omega)}^{2}).
		\end{split}
	\end{align}
	Observe that
	\begin{align*}
		|\int_{\Omega}\rho\mu u\cdot\nabla\phi dx \rvert&\leqslant  C(\rho^{*})\|\nabla\phi\|_{L^{2}(\Omega)}\|u\|_{L^{3}(\Omega)}\|\mu\|_{L^{6}(\Omega)}\\
		&\leqslant C(\rho^{*})\|\nabla u\|_{L^{2}(\Omega)}\|\mu\|_{H^{1}(\Omega)}\\
		&\leqslant  C(\rho^{*})\|\nabla u\|_{L^{2}(\Omega)}(1+\|\nabla\mu\|_{L^{2}(\Omega)})\\
		&\leqslant  \int_{\Omega}\frac{v(\phi)}{4}|\mathbb{D}u |^{2} dx+\frac{1}{4}\|\nabla\mu\|_{L^{2}(\Omega)}^{2}+C(\rho^{*}),
	\end{align*}
	implies that
	\begin{eqnarray}\label{3.22}
		&&\int_{\Omega}\frac{v(\phi)}{2}|\mathbb{D}u |^{2} dx+\frac{1}{2}\|\nabla\mu\|_{L^{2}(\Omega)}^{2}+\int_{\Omega}\rho\mu u\cdot\nabla\phi dx\nonumber\\
		&&\geqslant\int_{\Omega}\frac{v(\phi)}{2}|\mathbb{D}u |^{2} dx+\frac{1}{2}\|\nabla\mu\|_{L^{2}(\Omega)}^{2}-\int_{\Omega}\frac{v(\phi)}{4}|\mathbb{D}u|^{2} dx-\frac{1}{4}\|\nabla\mu\|_{L^{2}(\Omega)}^{2}
		-C(\rho^{*})\nonumber\\
		&&\geqslant\int_{\Omega}\frac{v(\phi)}{4}|\mathbb{D}u |^{2} dx+\frac{1}{4}\|\nabla\mu\|_{L^{2}(\Omega)}^{2}-C(\rho^{*})\nonumber\\
		&&\geqslant C_1(\|\nabla u\|_{L^{2}(\Omega)}^{2}+\|\nabla\mu\|_{L^{2}(\Omega)}^{2})-C(\rho^{*})
	\end{eqnarray}
	with $C_1=\min\{v_*,1\}.$	Taking similar argument, one finds that
	\begin{align*}
			\int_{\Omega}\frac{v(\phi)}{2}|\mathbb{D}u |^{2} dx+\frac{1}{2}\|\nabla\mu\|_{L^{2}(\Omega)}^{2}+\int_{\Omega}\rho\mu u\cdot\nabla\phi dx\leqslant\int_{\Omega}\frac{3v(\phi)}{4}|\mathbb{D}u |^{2} dx+\frac{3}{4}\|\nabla\mu\|_{L^{2}(\Omega)}^{2}+C(\rho^{*}).
	\end{align*}
	Combining \eqref{3.21} with \eqref{3.22}, one eventually obtains that
	\begin{align}\label{3.23}
		\begin{split}
			&\frac{d}{dt}\left\{\|\nabla u\|_{L^{2}(\Omega)}^{2}+\|\nabla \mu\|_{L^{2}(\Omega)}^{2}+1\right\}+\int_{\Omega}|\nabla\phi_t  |^{2} dx+\int_{\Omega}\rho| \phi_t  |^{2} dx+\int_{\Omega}\rho|u_t  |^{2} dx\\
			&\leqslant  C(\rho^{*}) (\|u\|_{L^{r}(\Omega)}^{\frac{4r}{r-6}}+1+\|\nabla\mu\|_{L^{2}(\Omega)}^{2}+\|\nabla u\|_{L^{2}(\Omega)}^{2})(1+\|\nabla\mu\|_{L^{2}(\Omega)}^{2}+\|\nabla u\|_{L^{2}(\Omega)}^{2}).
		\end{split}
	\end{align}
	So,
	\begin{align*}
		\begin{split}
			&\frac{d}{dt}\left\{\|\nabla u\|_{L^{2}(\Omega)}^{2}+\|\nabla \mu\|_{L^{2}(\Omega)}^{2}+1\right\}\\
			&\leqslant C(\rho^*) (\|u\|_{L^{r}(\Omega)}^{\frac{4r}{r-6}}+1+\|\nabla\mu\|_{L^{2}(\Omega)}^{2}+\|\nabla u\|_{L^{2}(\Omega)}^{2})(1+\|\nabla\mu\|_{L^{2}(\Omega)}^{2}+\|\nabla u\|_{L^{2}(\Omega)}^{2}).
		\end{split}
	\end{align*}
	According to Lemma \ref{lem3-3} and Gagliardo-Nirenberg inequality, one deduces from \eqref{3.23} and \eqref{3.1} that
	\begin{align}\label{3.24}
		\begin{split}
			&\sup_{t\in [0,T]}(\|\nabla u\|_{L^{2}(\Omega)}^{2}+\|\nabla \mu\|_{L^{2}(\Omega)}^{2})\\
			&\leqslant exp\left(C(\rho^*){\int_{0}^{T}(\|u\|_{L^{r}(\Omega)}^{\frac{4r}{r-6}}+1+\|\nabla\mu\|_{L^{2}(\Omega)}^{2}+\|\nabla u\|_{L^{2}(\Omega)}^{2}) dt}\right)\cdot\\
			&\quad(\|\nabla u_0\|_{L^{2}(\Omega)}^{2}+\|\nabla \mu_0\|_{L^{2}(\Omega)}^{2}+1)\\
			&\leqslant C(\rho^*).
		\end{split}
	\end{align}
	By combining \eqref{3.23} with \eqref{3.24}, we obtain
	\begin{align*}
		\begin{split}
			\sup_{t\in [0,T]}(\|\nabla u\|_{L^{2}(\Omega)}^{2}+\|\nabla\mu\|_{L^{2}(\Omega)}^{2})+\int_{0}^{T}\int_{\Omega}(|\nabla \phi_{t} \rvert^{2}+\rho\phi_t^{2}+\rho | u_{t}\rvert^{2}) dx dt
			\leqslant C(\rho^*).
		\end{split}
	\end{align*}
	Using Lemma \ref{lem3-1} and Lemma \ref{lem3-3} again, one obtains from \eqref{3.23} that
	\begin{align*}
			&\sup_{t \in [0,T]}\big(\|\nabla u\|_{L^{2}(\Omega)}^{2}+\|\nabla \mu\|_{L^{2}(\Omega)}^{2}+\|\phi\|_{W^{2,6}(\Omega)}^{2} \big)+ \int_{0}^{T}(\| \phi_t\|_{H^{1}(\Omega)}^{2}+\| u_t \|_{L^{2}(\Omega)}^{2}+\|u\|_{H^{2}(\Omega)}^{2}) dt\leqslant C(\rho^*).
	\end{align*}
	The proof of Lemma \ref{lem3-4} is completed.
\end{proof}

\section{Strong solution: global well-posedness}\label{S4}

In this section, we prove Theorem \ref{thm1.2} on the existence of global strong solutions to the problem \eqref{1.1}-\eqref{1.4}. To extend the short time strong solution to be a global one, it is essential to establish the uniform estimates. Thus, we devote ourselves to obtain the global-in-time a priori estimates, which is analogous to the blowup criterion established in Section \ref{S3}. To make the content self-contained, a detailed proof is given as follows. The proof of existence relies on the abstract bootstrap argument and suitable {\sl a priori} energy estimates.

\begin{lemma}\label{lem4-1}
	Under the assumption of Theorem \ref{thm1.2}, let $(\rho,u,\phi,\mu)$ be a strong solution to the problem \eqref{1.1}-\eqref{1.4} on $(0,T].$ Then there exists a positive constant $C,$ depending only on $|\Omega|, \ \|\phi_0\|_{L^2(\Omega)}, \  \|\mu_0\|_{L^2(\Omega)},$ such that
	\begin{equation}\label{4.1}
		0<\frac{\rho_*}{C}\leqslant \rho(x,t)\leqslant C\epsilon_0
	\end{equation}
	and
	\begin{align}\label{4.2}
			\sup_{t\in [0,T]}\int_{\Omega}\big(\frac{\rho u^{2}}{2}+\frac{\rho(\phi^{2}-1)^{2}}{4}+\frac{|\nabla \phi|^{2}}{2}\big) dx+\int_{0}^{T}\int_{\Omega}\big(v(\phi)|{\mathbb D} u|^{2}+|\nabla \mu |^{2}\big) dx dt\leqslant  C\epsilon_0,
	\end{align}
	where $\epsilon_0$ is a small positive constant in \eqref{1.12}.
\end{lemma}

\begin{proof}
	It is easy to get \eqref{4.1}. Taking same argument for \eqref{3.8}, one arrives at
	\begin{align*}
		\begin{split}
			&\int_{\Omega}\big(\frac{1}{2}\rho|u|^{2}+\frac{1}{2}|\nabla\phi|^{2}+\rho\Psi(\phi)\big) dx
			+\int_{0}^{t}\int_{\Omega}\big(v(\phi)| \mathbb{D}u|^{2}+|\nabla\mu|^{2}\big) dxd\tau \\
			&=\int_{\Omega}\frac{1}{2}\rho_0 |u_0 |^{2} + \frac{1}{2}|\nabla\phi_0 |^{2}+\rho_0 \Psi(\phi_0) dx.
		\end{split}
	\end{align*}
	Next, one takes into account
	$$\mu_0=-\frac{\Delta \phi_0}{\rho_0}+\Psi^\prime(\phi_0)$$
	and finds that
	
	$$\int_{\Omega}|\nabla\phi_0|^2 dx+\int_{\Omega}\rho_0(\phi_{0}^{2}-1)\phi_{0}^{2} dx = \int_{\Omega}\rho_{0}\mu_{0}\phi_{0} dx.$$
	So,
	\begin{eqnarray*}
		&&\int_{\Omega}\rho_0 (\phi_{0}^{2}-1)^{2} dx+\int_{\Omega}|\nabla\phi_0 |^{2} dx=\int_{\Omega}\rho_0\mu_0\phi_0 dx-\int_{\Omega}\rho_0(\phi_{0}^{2}-1) dx\\
		&&\leqslant\epsilon_0\|\mu_0\|_{L^{2}(\Omega)}\|\phi_0\|_{L^{2}(\Omega)}
		+\frac{1}{2}\int_{\Omega}\rho_0(\phi_{0}^{2}-1)^{2} dx+\frac{1}{2}\int_{\Omega}\rho_0  dx\\
		&&\leqslant C\epsilon_0 +\frac{1}{2}\int_{\Omega}\rho_0(\phi_{0}^{2}-1)^{2} dx,
	\end{eqnarray*}
	which implies that
	\begin{eqnarray}\label{4.3}
		&&\int_{\Omega}\frac{1}{4}\rho_0 (\phi_{0}^{2}-1)^{2} dx+\int_{\Omega}\frac{1}{2}|\nabla\phi_0 |^{2} dx
		\leqslant C\epsilon_0.
	\end{eqnarray}
	Moreover, one gets that
	\begin{eqnarray}\label{4.4}
		\frac{1}{2}\int_{\Omega}\rho_0 |u_0 |^{2} dx\leqslant \epsilon_0\|u_0\|^{2}_{L^{2}(\Omega)}\leqslant \epsilon_0\|\nabla u_0\|_{L^{2}(\Omega)}^{2}\leqslant \epsilon_0^{3}.
	\end{eqnarray}
	Thus, one deduces from 	\eqref{4.3} and \eqref{4.4} that
	\begin{eqnarray}\label{4-0}
		&&\int_{\Omega}\big(\frac{1}{2}\rho_0 |u_0 |^{2} + \frac{1}{2}|\nabla\phi_0 |^{2}+\rho_0 \Psi(\phi_0)\big) dx\nonumber\\
		&&=\int_{\Omega}\frac{1}{2}\rho_0 |u_0 |^{2} dx+\int_{\Omega}\frac{1}{4}\rho_0 (\phi_{0}^{2}-1)^{2} dx+\int_{\Omega}\frac{1}{2}|\nabla\phi_0 |^{2} dx\nonumber\\
		&&\leqslant C\epsilon_0.
	\end{eqnarray}
	Therefore, one obtains that
	\begin{align*}
			\int_{\Omega}\big(\frac{1}{2}\rho|u|^{2}+\frac{1}{2}|\nabla\phi|^{2}+\rho\Psi(\phi)\big) dx+\int_{0}^{t}\int_{\Omega}\big(v(\phi)| \mathbb{D}u|^{2}+|\nabla\mu|^{2}\big) dxdt
			\leqslant C\epsilon_0.
	\end{align*}
	This completes the proof of Lemma \ref{lem4-1}.
\end{proof}

Next, we aim to prove the following key a priori estimates, which mainly concerns the uniform bound.

\begin{lemma}\label{lem4-2}
	Under the assumption of Theorem \ref{thm1.2}, let $(\rho,u,\phi,\mu)$ be a strong solution to the problem \eqref{1.1}-\eqref{1.4} on $(0,T].$
	Then there exists a positive constant $C,$ depending only on $|\Omega|,\|\phi_0\|_{L^2(\Omega)}, \|\mu_0\|_{L^2(\Omega)},$ such that
	\begin{eqnarray}
		&&\sup_{t\in [0,T]}\|\phi\|_{H^{1}(\Omega)}\leqslant  C\epsilon_0,\label{4.5}\\
		&&\|\mu\|_{H^{1}(\Omega)}\leqslant  C\|\nabla \mu\|_{L^{2}(\Omega)},\label{4.6}\\
		&&\|\phi\|_{W^{2,6}(\Omega)}\leqslant  C\epsilon_0\|\nabla\mu\|_{L^{2}(\Omega)}, \ \label{4.7} \\
		&&\|u\|_{H^{2}(\Omega)}\leqslant  C\big(\epsilon_0^2\|\nabla \mu\|_{L^{2}(\Omega)}^{\frac{5}{4}}
		+\epsilon_0^\frac{1}{2}\|\sqrt{\rho} u_t \|_{L^{2}(\Omega)}+\epsilon_0^2\|\nabla u\|_{L^{2}(\Omega)}^{3}\big), \ \  \label{4.8}
	\end{eqnarray}
	where $\epsilon_0$ is a small positive constant in \eqref{1.12}.
\end{lemma}

\begin{proof}
	Firstly, it is deduced from Lemma \ref{lem4-1} that
	$$\int_{\Omega}(\rho\phi^{4}+\rho) dx\leqslant C\epsilon_0+2\int_{\Omega}\rho\phi^{2} dx\leqslant C\epsilon_0+\int_{\Omega}\frac{1}{2}\rho\phi^{4} dx,$$
	and so
	$$\int_{\Omega}\rho\phi^{4} dx\leqslant C\epsilon_0.$$
	Thus, one gets that
	\begin{equation}\label{4.9}\int_{\Omega}\rho|\phi \rvert  dx\leqslant \big(\int_{\Omega}\rho|\phi |^{4} dx\big)^{\frac{1}{4}}\big(\int_{\Omega}\rho dx\big)^{\frac{3}{4}}\leqslant  C\epsilon_0.
	\end{equation}
	Combining \eqref{4.2} and \eqref{4.9}, one obtains from Lemma \ref{lem2-2} that
	$$\sup_{t\in [0,T]}\|\phi\|_{H^{1}(\Omega)}\leqslant  C\epsilon_0.$$
	
	Secondly, one deduces from Lemma \ref{lem2-2} that
	\begin{align}\label{4.10}
			\|\mu\|_{L^{2}}^{2}\leqslant C\big(\|\nabla\mu\|_{L^{2}(\Omega)}^{2}+(\int_{\Omega}\rho|\mu | dx)^{2}\big)\leqslant C\|\nabla \mu\|_{L^{2}(\Omega)}^{2}+C\epsilon_0^{2}\|\mu\|_{L^{2}(\Omega)}^{2}.
	\end{align}
	Taking $\epsilon_0$ small enough such that $C\epsilon_0^{2}<\frac12,$ one gets \eqref{4.6} immediately.
	
	Thirdly, applying Lemma \ref{lem2-3} to $\eqref{1.1}_5$,  one obtains from \eqref{4.6} that
	\begin{align*}
		\begin{split}
			\|\phi\|_{W^{2,6}(\Omega)}&\leqslant  C(\Omega)\|\rho\mu+\rho\Psi^{\prime}(\phi)\|_{L^{6}(\Omega)}\\
			&\leqslant  C(\Omega)\epsilon_0\big(\|\mu\|_{L^{6}(\Omega)}+\|\phi^{3}\|_{L^{6}(\Omega)}+\|\phi\|_{L^{6}(\Omega)}\big)\\
			&\leqslant C(\Omega)\epsilon_0(\|\mu\|_{H^{1}(\Omega)}+\|\phi^{3}\|_{L^{2}(\Omega)}+\|\nabla\phi \phi^{2}\|_{L^{2}(\Omega)}+\|\phi\|_{L^{6}(\Omega)})\\
			&\leqslant  C\epsilon_0\|\nabla \mu\|_{L^{2}(\Omega)}+C(\Omega)\epsilon_0\big(\|\phi\|_{L^{2}(\Omega)}^{2}\|\phi\|_{W^{2,6}(\Omega)}+\|\nabla\phi\|_{L^{2}(\Omega)}\|\phi\|_{L^{2}(\Omega)}\|\phi\|_{W^{2,6}(\Omega)}+\|\phi\|_{W^{2,6}(\Omega)}\big)\\
			&\leqslant  C\epsilon_0\|\nabla \mu\|_{L^{2}(\Omega)}+C\epsilon_0\|\phi\|_{W^{2,6}(\Omega)},
		\end{split}
	\end{align*}
	Taking $\epsilon_0$ small enough such that $C\epsilon_0^{2}<\frac12,$ one gets \eqref{4.7} directly.
	
	Finally,  one applies Lemma \ref{lem2-4} on $\eqref{1.1}_2$ to arrive that
	\begin{align*}
		\begin{split}
			\|u\|_{H^{2}(\Omega)}
			&\leqslant C(\Omega)\|-\Delta\phi\nabla\phi-\nabla^{2}\phi\cdot\nabla\phi-\rho u_t -\rho(u\cdot\nabla)u\|_{L^{2}(\Omega)}\\
			&\leqslant C(\Omega)(\|\Delta\phi\|_{L^{6}(\Omega)}\|\nabla\phi\|_{L^{3}(\Omega)}
			+\|\nabla^{2}\phi\|_{L^{6}(\Omega)}\|\nabla\phi\|_{L^{3}(\Omega)}+\epsilon_0^{\frac{1}{2}}\|\sqrt{\rho} u_t \|_{L^{2}(\Omega)}+\epsilon_0\|\nabla u\|_{L^{3}(\Omega)}\|u\|_{L^{6}(\Omega)})\\
			&\leqslant C(\Omega)(\|\phi\|_{W^{2,6}(\Omega)}\|\nabla\phi\|_{L^{2}(\Omega)}^{\frac{3}{4}}\|\phi\|_{W^{2,6}(\Omega)}^{\frac{1}{4}}+\epsilon_{0}^{\frac{1}{2}}\|\sqrt{\rho}u_t\|_{L^{2}(\Omega)}+\epsilon_{0}\|\nabla u\|_{L^{2}(\Omega)}^{\frac{3}{2}}\|u\|_{H^{2}(\Omega)}^{\frac{1}{2}}) \\
			&\leqslant  C\epsilon_{0}^{\frac{3}{4}}\|\phi\|_{W^{2,6}(\Omega)}^{\frac{5}{4}}
			+C\epsilon_0^{\frac{1}{2}}\|\sqrt{\rho} u_t \|_{L^{2}(\Omega)}+C\epsilon_0\|\nabla u\|_{L^{2}(\Omega)}^{3}\|u\|_{H^{2}(\Omega)}^{\frac{1}{2}}\\
			&\leqslant \eta \|u\|_{H^{2}(\Omega)}+ C (\epsilon_0^{2}\|\nabla\mu\|_{L^{2}(\Omega)}^{\frac{5}{4}}
			+\epsilon_0^{\frac{1}{2}}\|\sqrt{\rho} u_t\|_{L^{2}(\Omega)}+\epsilon_0^{2}\|\nabla u\|_{L^{2}(\Omega)}^{3}),
		\end{split}
	\end{align*}
	holds for any fixed $\eta \in (0,1).$ Selecting suitable small $\eta \in (0,1),$ one obtains that
	\begin{equation*}\|u\|_{H^{2}(\Omega)}\leqslant
		C\big(\epsilon_0^{2}\|\nabla\mu\|_{L^{2}(\Omega)}^{\frac{5}{4}}+\epsilon_0^{\frac{1}{2}}\|\sqrt{\rho} u_t\|_{L^{2}(\Omega)}+\epsilon_0^{2}\|\nabla u\|_{L^{2}(\Omega)}^{3}\big).
	\end{equation*}
	So the desired estimate \eqref{4.5} is proved.
\end{proof}

The following Lemma \ref{lem4-4} is concerned with the estimates on the higher-order derivatives of the velocity, the difference of fluids concentrations and the chemical potential.

\begin{lemma}\label{lem4-4}
	Under the assumption of Theorem \ref{thm1.2}, let $(\rho,u,\phi,\mu)$ be a strong solution to the problem \eqref{1.1}-\eqref{1.4} on $(0,T].$
	Then there exists a positive constant $C,$ depending only on $|\Omega|$, $v_*$, $v^*$, $\|\phi_0\|_{L^2(\Omega)}$, $\|\mu_0\|_{L^2(\Omega)}$, such that
	\begin{align*}
\sup_{0\leqslant  t\leqslant T}\big(\|\nabla u\|^{2}_{L^{2}(\Omega)}+\|\nabla \mu\|^{2}_{L^{2}(\Omega)}+\|\phi\|^{2}_{W^{2,6}(\Omega)} \big)+ \int_{0}^{T}\big(\| \phi_t\|_{H^{1}(\Omega)}^{2}+\|  u_t \|_{L^{2}(\Omega)}^{2}+\|u\|_{H^{2}(\Omega)}^{2}\big) dt
			\leqslant C\epsilon_0,
	\end{align*}
	where $\epsilon_0$ is a small positive constant in \eqref{1.12}.
\end{lemma}

\begin{proof}
	Recalling the equation \eqref{3.19}, we write it down again here.
	\begin{align}\label{4.11}
		&\frac{d}{dt}\left\{\int_{\Omega}\frac{v(\phi)}{2}|\mathbb{D}u |^{2} dx+\frac{1}{2}\|\nabla\mu\|_{L^{2}(\Omega)}^{2}
		+\int_{\Omega}\rho\mu u\cdot\nabla\phi dx\right\}\nonumber\\
		&\quad+\int_{\Omega}|\nabla\phi_t  |^{2} dx+\int_{\Omega}\rho| \phi_t  |^{2} dx+\int_{\Omega}\rho| u_t  |^{2} dx
		+3\int_{\Omega}\rho\phi^{2}|\phi_t \rvert^{2} dx\nonumber\\
		&=\underbrace{-\int_{\Omega}\rho(u\cdot\nabla)u \cdot u_t   dx}_{I_1}+\underbrace{2\int_{\Omega}\rho\mu\nabla\phi\cdot u_t  dx}_{I_2}
		+\underbrace{\int_{\Omega}-\rho\nabla\Psi(\phi)\cdot u_t dx}_{I_3}\nonumber\\
		&\quad+\underbrace{\int_{\Omega}v^{\prime}(\phi) \phi_t |\mathbb{D}u |^{2}dx}_{I_4}
		+\underbrace{\int_{\Omega}-\rho\Psi^{\prime\prime} (\phi) \phi_t  u\cdot\nabla\phi  dx}_{I_5}
		+\underbrace{\int_{\Omega}-\rho\Psi^{\prime}(\phi)u\cdot\nabla\phi_t dx}_{I_6}\nonumber\\
		&\quad+\underbrace{\int_{\Omega}\rho \phi_t  u\cdot\nabla\mu  dx}_{I_7}
		+\underbrace{\int_{\Omega}\rho(u\cdot\nabla\mu)(u\cdot\nabla\phi) dx}_{I_8}+\underbrace{\int_{\Omega}\rho\mu u\cdot\nabla(u\cdot\nabla\phi) dx}_{I_9}\nonumber\\
		&\quad+\underbrace{2\int_{\Omega}\rho\mu u\cdot\nabla\phi_t   dx}_{I_{10}}
		+\underbrace{2\int_{\Omega}-\rho u\cdot\nabla\phi \phi_t   dx}_{I_{11}}+\underbrace{2\int_{\Omega}-\nabla\mu\cdot\nabla\phi_t dx}_{I_{12}}.
	\end{align}
	Now, we exploit Lemma \ref{lem4-2}, Sobolev inequality, H$\ddot{\text{o}}$lder inequality and Young's inequality to estimate each term on the right hand of \eqref{4.11} as follows:
	\begin{align*}
		|I_1|&=|\int_{\Omega}\rho (u\cdot \nabla)u\cdot u_t  dx |\\
		&\leqslant  C\epsilon_{0}^{\frac{1}{2}}\|\sqrt{\rho} u_t \|_{L^{2}(\Omega)}\|u\|_{L^{6}(\Omega)}\|\nabla u\|_{L^{3}(\Omega)}\\
		&\leqslant C\epsilon_{0}^{\frac{1}{2}}\|\sqrt{\rho}u_t\|_{L^{2}(\Omega)}\|\nabla u\|_{L^{2}(\Omega)}^{\frac{3}{2}}\|u\|_{H^{2}(\Omega)}^{\frac{1}{2}} \\
		&\leqslant\eta\|\sqrt{\rho} u_t \|_{L^{2}(\Omega)}^{2}+C\epsilon_0\|\nabla u\|_{L^{2}(\Omega)}^{3}\| u\|_{H^{2}(\Omega)}\\
		&\leqslant\eta \|\sqrt{\rho} u_t\|_{L^{2}(\Omega)}^{2}+C\epsilon_0\|\nabla u\|_{L^{2}(\Omega)}^{3}(\epsilon_{0}^{2}\|\nabla\mu\|_{L^{2}(\Omega)}^{\frac{5}{4}}
		+\epsilon_{0}^{\frac{1}{2}}\|\sqrt{\rho}u_t\|_{L^{2}(\Omega)}+\epsilon_{0}^{2}\|\nabla u\|_{L^{2}(\Omega)}^{3}) \\
		&\leqslant\eta\|\sqrt{\rho} u_t \|_{L^{2}(\Omega)}^{2}+C\epsilon_0^{3}\|\nabla u\|_{L^{2}(\Omega)}^{6}+C\epsilon_0^{3}\|\nabla \mu\|_{L^{2}(\Omega)}^{\frac{5}{2}},\\
		|I_2|&=|2\int_{\Omega}\rho\mu\nabla\phi\cdot u_t dx| \\
		&\leqslant  C\epsilon_{0}^{\frac{1}{2}}\|\sqrt{\rho} u_t \|_{L^{2}(\Omega)}\|\mu\|_{L^{6}(\Omega)}\|\nabla\phi\|_{L^{3}(\Omega)}\\
		&\leqslant C\epsilon_{0}^{\frac{1}{2}}\|\sqrt{\rho}u_t\|_{L^{2}(\Omega)}\|\mu\|_{H^{1}(\Omega)}
		\|\nabla\phi\|_{L^{2}(\Omega)}^{\frac{3}{4}}\|\phi\|_{W^{2,6}(\Omega)}^{\frac{1}{4}} \\
		&\leqslant  C\epsilon_{0}^{\frac{3}{2}}\|\sqrt{\rho} u_t \|_{L^{2}(\Omega)}\|\nabla\mu\|_{L^{2}(\Omega)}^{\frac{5}{4}}\\
		&\leqslant\eta\|\sqrt{\rho} u_t \|_{L^{2}(\Omega)}^{2}+C\epsilon_0^{3}\|\nabla\mu\|_{L^{2}(\Omega)}^{\frac{5}{2}},\\
		|I_3|&=|\int_{\Omega}\rho\nabla\Psi(\phi)\cdot u_t dx | \\
		&\leqslant  C\epsilon_0^{\frac{1}{2}}\|\Psi^{\prime}(\phi)\|_{L^{\infty}(\Omega)}\|\nabla\phi\|_{L^{2}(\Omega)}\|\sqrt{\rho} u_t \|_{L^{2}(\Omega)}\\
		&\leqslant C\epsilon_{0}\|\sqrt{\rho}u_t\|_{L^{2}(\Omega)}(\|\phi\|_{L^{\infty}(\Omega)}^{3}+\|\phi\|_{L^{\infty}(\Omega)})\|\phi\|_{W^{2,6}(\Omega)}^{\frac{1}{2}}\\
		&\leqslant C\epsilon_0\|\sqrt{\rho} u_t\|_{L^{2}(\Omega)}\|\phi\|_{L^{2}(\Omega)}^{\frac{3}{2}}\|\phi\|_{W^{2,6}(\Omega)}^{2}
		+C\epsilon_{0}^{\frac{3}{2}}\|\sqrt{\rho}u_t\|_{L^{2}(\Omega)}\|\phi\|_{W^{2,6}(\Omega)} \\
		&\leqslant\eta\|\sqrt{\rho}u_t\|_{L^{2}(\Omega)}^{2}+C\epsilon_0^{5}\|\phi\|_{W^{2,6}(\Omega)}^{4}+C\epsilon_0^{3}\|\phi\|_{W^{2,6}(\Omega)}^{2}\\
		&\leqslant\eta\|\sqrt{\rho}u_t\|_{L^{2}(\Omega)}^{2}+C\epsilon_0^{5}(1+\epsilon_0^{4}\|\nabla\mu\|_{L^{2}(\Omega)}^{2})\|\nabla\mu\|_{L^{2}(\Omega)}^{2},\\
		|I_4 |&=|\int_{\Omega}v^{\prime}(\phi)\phi_t|\mathbb{D}u |^{2} dx |   \\
		&\leqslant  C\| \phi_t \|_{L^{6}(\Omega)}\|\nabla u\|_{L^{3}(\Omega)}\|\nabla u\|_{L^{2}(\Omega)}\\
		&\leqslant C(\|\nabla\phi_t\|_{L^{2}(\Omega)}+\epsilon_{0}^{\frac{1}{2}}\|\sqrt{\rho}\phi_t\|_{L^{2}(\Omega)})
		\|\nabla u\|_{L^{2}(\Omega)}^{\frac{3}{2}}\|u\|_{H^{2}(\Omega)}^{\frac{1}{2}} \\
		&\leqslant\eta(\|\nabla \phi_t\|_{L^{2}(\Omega)}^{2}+\|\sqrt{\rho}\phi_t\|_{L^{2}(\Omega)}^{2})+C\|\nabla u\|_{L^{2}(\Omega)}^{3}\|u\|_{H^{2}(\Omega)}\\
		&\leqslant\eta(\|\nabla \phi_t\|_{L^{2}(\Omega)}^{2}+\|\sqrt{\rho}\phi_t\|_{L^{2}(\Omega)}^{2}+\|\sqrt{\rho}u_t\|_{L^{2}(\Omega)}^{2})
		+C\epsilon_0\|\nabla u\|_{L^{2}(\Omega)}^{6}+C\epsilon_0^{3}\|\nabla\mu\|_{L^{2}(\Omega)}^{\frac{5}{2}},\\
		|I_5 |&=|\int_{\Omega}\rho\Psi^{\prime\prime}(\phi)\phi_t u\cdot \nabla\phi dx |  \\
		&\leqslant  C\epsilon_0\|\nabla\phi\|_{L^{2}(\Omega)}\|u\|_{L^{6}(\Omega)}\| \phi_t \|_{L^{6}(\Omega)}\|\Psi^{\prime\prime} (\phi)\|_{L^{6}}\\
		&\leqslant C\epsilon_{0}^{2}(\|\nabla\phi_t\|_{L^{2}(\Omega)}
		+\epsilon_{0}^{\frac{1}{2}}\|\sqrt{\rho}\phi_t\|_{L^{2}(\Omega)})\|\nabla u\|_{L^{2}(\Omega)}(1+\|\phi\|_{L^{12}(\Omega)}^{2}) \\
		&\leqslant  C\epsilon_0^{2}(\|\nabla\phi_t\|_{L^{2}(\Omega)}
		+\epsilon_0^{\frac{1}{2}}\|\sqrt{\rho}\phi_t\|_{L^{2}(\Omega)})\|\nabla u\|_{L^{2}(\Omega)}
		(1+\|\phi\|_{L^{2}(\Omega)}^{\frac{7}{6}}\|\phi\|_{W^{2,6}(\Omega)}^{\frac{5}{6}})\\
		&\leqslant\eta(\|\nabla\phi_t\|_{L^{2}(\Omega)}^{2}+\|\sqrt{\rho}\phi_t\|_{L^{2}(\Omega)}^{2})
		+C\epsilon_0^{4}\|\nabla u\|_{L^{2}(\Omega)}^{2}(1+\epsilon_0^{4}\|\nabla \mu\|_{L^{2}(\Omega)}^{\frac{5}{3}}),\\
		|I_6 |& = |\int_{\Omega}\rho\Psi^{\prime}(\phi)u\cdot\nabla\phi_t dx|\\
		&\leqslant  C\epsilon_0\|\Psi^{\prime}(\phi)\|_{L^{3}(\Omega)}\|u\|_{L^{6}(\Omega)}\|\nabla\phi_t \|_{L^{2}(\Omega)}\\
		&\leqslant \eta\|\nabla\phi_t\|_{L^{2}(\Omega)}^{2}+C\epsilon_{0}^{2}\|\nabla u\|_{L^{2}(\Omega)}^{2}\|\phi^{3}-\phi\|_{L^{3}(\Omega)}^{2} \\
		&\leqslant\eta\|\nabla\phi_t\|_{L^{2}(\Omega)}^{2}+C\epsilon_0^{2}\|\nabla u\|_{L^{2}(\Omega)}^{2}(\|\phi\|_{L^{9}(\Omega)}^{3}+\|\phi\|_{L^{3}(\Omega)})^{2}\\
		&\leqslant\eta\|\nabla\phi_t\|_{L^{2}(\Omega)}^{2}
		+C\epsilon_0^{2}\|\nabla u\|_{L^{2}(\Omega)}^{2}(\|\phi\|_{L^{6}(\Omega)}^{\frac{11}{2}}\|\phi\|_{W^{2,6}(\Omega)}^{\frac{1}{2}}+\|\phi\|_{H^{1}(\Omega)}^{2})\\
		&\leqslant\eta\|\nabla\phi_t\|_{L^{2}(\Omega)}^{2}+C\epsilon_0^{4}\|\nabla u\|_{L^{2}(\Omega)}^{2}(1+\|\nabla\mu\|_{L^{2}(\Omega)}^{2}),\\
		|I_7 |&= |\int_{\Omega}\rho\phi_t u\cdot \nabla\mu dx|\\
		&\leqslant  C\epsilon_0\| \phi_t \|_{L^{6}(\Omega)}\|u\|_{L^{3}(\Omega)}\|\nabla\mu\|_{L^{2}(\Omega)}\\
		&\leqslant C\epsilon_0(\|\nabla \phi_t\|_{L^{2}(\Omega)}
		+\epsilon_0^{\frac{1}{2}}\|\sqrt{\rho}\phi_t\|_{L^{2}(\Omega)})\|\nabla u\|_{L^{2}(\Omega)}\|\nabla \mu\|_{L^{2}(\Omega)}\\
		&\leqslant\eta(\|\nabla \phi_t\|^{2}_{L^{2}(\Omega)}+\|\sqrt{\rho}\phi_t\|_{L^{2}(\Omega)}^{2})
		+C\epsilon_0^{2}\|\nabla u\|_{L^{2}(\Omega)}^{2}\|\nabla \mu\|_{L^{2}(\Omega)}^{2},\\
		|I_8 |& = |\int_{\Omega}\rho(u\cdot\nabla\mu)(u\cdot\nabla\phi) dx|\\
		&\leqslant  C\epsilon_0\|u\|_{L^{6}(\Omega)}^{2}\|\nabla\mu\|_{L^{2}(\Omega)}\|\nabla\phi\|_{L^{6}(\Omega)}\\
		&\leqslant  C\epsilon_0\|\nabla u\|_{L^{2}(\Omega)}^{2}\|\nabla \mu\|_{L^{2}(\Omega)}
		\|\nabla \phi\|_{L^{2}(\Omega)}^{\frac{1}{2}}\|\phi\|_{W^{2,6}(\Omega)}^{\frac{1}{2}}\\
		&\leqslant  C\epsilon_0^{2}\|\nabla  u\|_{L^{2}(\Omega)}^{2}(1+\|\nabla \mu\|_{L^{2}(\Omega)}^{2}),\\
		|I_9 |&= |\int_{\Omega}\rho\mu u\cdot\nabla(u\cdot\nabla\phi) dx|\\
		&\leqslant  C\epsilon_0\|\mu\|_{L^{6}(\Omega)}\| u\|_{L^{6}(\Omega)}(\|\nabla u\|_{L^{2}(\Omega)}\|\nabla \phi\|_{L^{6}(\Omega)}
		+\|u\|_{L^{2}(\Omega)}\|\nabla^{2}\phi\|_{L^{6}(\Omega)})\\
		&\leqslant  C\epsilon_0\|\mu\|_{H^{1}(\Omega)}\|\nabla u\|_{L^{2}(\Omega)}^{2}
		(\|\nabla \phi\|_{L^{2}(\Omega)}^{\frac{1}{2}}\|\phi\|_{W^{2,6}(\Omega)}^{\frac{1}{2}}+\|\phi\|_{W^{2,6}(\Omega)})\\
		&\leqslant C\epsilon_0^{2}\|\nabla u\|_{L^{2}(\Omega)}^{2}(1+\|\nabla \mu\|_{L^{2}(\Omega)}^{2}),\\
		|I_{10}|&= |2\int_{\Omega}\rho\mu u\cdot\nabla\phi_t dx|\\
		&\leqslant  C\epsilon_0\|\mu\|_{L^{6}(\Omega)}\|u\|_{L^{3}(\Omega)}\|\nabla\phi_t \|_{L^{2}(\Omega)}\\
		&\leqslant  C\epsilon_0\|\nabla \phi_t\|_{L^{2}(\Omega)}\|\nabla u\|_{L^{2}(\Omega)}\| \mu\|_{H^{1}(\Omega)}\\
		&\leqslant \eta\|\nabla \phi_t\|_{L^{2}(\Omega)}^{2}+C\epsilon_0^{2}\|\nabla  u\|_{L^{2}(\Omega)}^{2}\|\nabla \mu\|_{L^{2}(\Omega)}^{2},\\
		|I_{11}|&= |2\int_{\Omega}\rho u\cdot\nabla\phi\phi_t dx|\\
		&\leqslant   C\epsilon_0\| \phi_t \|_{L^{6}(\Omega)}\|\nabla\phi\|_{L^{3}(\Omega)}\|u\|_{L^{2}(\Omega)}\\
		&\leqslant C\epsilon_0(\|\nabla\phi_t\|_{L^{2}(\Omega)}+\epsilon_0^{\frac{1}{2}}\|\sqrt{\rho}\phi_t\|_{L^{2}(\Omega)})
		\|\nabla\phi\|_{L^{2}(\Omega)}^{\frac{3}{4}}\|\phi\|_{W^{2,6}(\Omega)}^{\frac{1}{4}}\|\nabla u\|_{L^{2}(\Omega)}\\
		&\leqslant\eta(\|\nabla \phi_t\|_{L^{2}(\Omega)}^{2}+\|\sqrt{\rho}\phi_t\|_{L^{2}(\Omega)}^{2})
		+C\epsilon_0^{4}\|\nabla u\|_{L^{2}(\Omega)}^{2}(1+\|\nabla\mu\|_{L^{2}(\Omega)}^{2}),\\
		|I_{12} |&= |2\int_{\Omega}\nabla\mu\cdot\nabla\phi_t dx|\\
		&\leqslant  C\|\nabla\mu\|_{L^{2}(\Omega)}\|\nabla\phi_t \|_{L^{2}(\Omega)}\\
		&\leqslant  \eta\|\nabla\phi_t \|_{L^{2}(\Omega)}^{2}+C\|\nabla\mu\|_{L^{2}(\Omega)}^{2}
	\end{align*}
	for any $\eta\in (0,1).$ Collecting all the above estimate and taking suitable small $\eta\in (0,1)$, one obtains that
	\begin{align}\label{4.12}
		\begin{split}
			&\frac{d}{dt}\left\{\int_{\Omega}\frac{v(\phi)}{2}|\mathbb{D}u |^{2} dx+\frac{1}{2}\|\nabla\mu\|_{L^{2}(\Omega)}^{2}+\int_{\Omega}\rho\mu u\cdot\nabla\phi  dx\right\}\\
			&\quad+\int_{\Omega}|\nabla\phi_t  |^{2} dx+\int_{\Omega}\rho| \phi_t  |^{2} dx+\int_{\Omega}\rho|u_t  |^{2}dx\\
			&\leqslant C (\epsilon_0+\|\nabla u\|_{L^{2}(\Omega)}^{2}+\|\nabla \mu\|_{L^{2}(\Omega)}^{2})(\epsilon_0\|\nabla u\|_{L^{2}(\Omega)}^{4}
			+\|\nabla\mu\|_{L^{2}(\Omega)}^{2}+\|\nabla u\|_{L^{2}(\Omega)}^{2})+C\|\nabla \mu\|_{L^{2}(\Omega)}^{2}.
		\end{split}
	\end{align}
	Thanks to Lemma \ref{lem4-2} and Gagliardo-Nirenberg inequality, one observes that
	\begin{eqnarray*}
		&|\int_{\Omega}\rho\mu u\cdot\nabla\phi dx \rvert&\leqslant  C\epsilon_0^{\frac{1}{2}}\|\nabla\phi\|_{L^{3}(\Omega)}\|\sqrt{\rho}u\|_{L^{2}(\Omega)}\|\mu\|_{L^{6}(\Omega)}\\
		&&\leqslant  C\epsilon_0\|\mu\|_{H^{1}(\Omega)}\|\nabla\phi\|_{L^{2}(\Omega)}^{\frac{3}{4}}\|\phi\|_{W^{2,6}(\Omega)}^{\frac{1}{4}}\\
		&&\leqslant C\epsilon_0^{2}\|\nabla\mu\|_{L^{2}(\Omega)}^{\frac{5}{4}}.
	\end{eqnarray*}
	In turn, this implies that
	\begin{align}\label{4.13}
		\begin{split}
			&\int_{\Omega}\frac{v(\phi)}{2}|\mathbb{D}u |^{2} dx+\frac{1}{2}\|\nabla\mu\|_{L^{2}(\Omega)}^{2}+\int_{\Omega}\rho\mu u\cdot\nabla\phi dx\\
			&\geqslant\int_{\Omega}\frac{v(\phi)}{2}|\mathbb{D}u |^{2} dx
			+\frac{1}{2}\|\nabla\mu\|_{L^{2}(\Omega)}^{2}-C\epsilon_0^{2}\|\nabla\mu\|_{L^{2}(\Omega)}^{\frac{5}{4}}\\
			&\geqslant \int_{\Omega}\frac{v(\phi)}{2}|\mathbb{D}u |^{2} dx+(\frac{1}{2}-C\epsilon_0^{2})\|\nabla\mu\|_{L^{2}(\Omega)}^{2}-C\epsilon_0^{2}\\
			&\geqslant \frac{\sqrt{2}v_*}{4}\|\nabla u\|_{L^{2}(\Omega)}^{2}+  (\frac{1}{2}-C\epsilon_0^{2})\|\nabla\mu\|_{L^{2}(\Omega)}^{2}-C\epsilon_0^{2} .
		\end{split}
	\end{align}
	Combining \eqref{4.12} with \eqref{4.13} and taking $\epsilon_0$ small enough such that
	$$C\epsilon_0^{2}<\frac{1}{2},$$
	we eventually obtain that
	\begin{align}\label{4.14}
		\begin{split}
			&\frac{d}{dt}\left\{\epsilon_0+\|\nabla u\|_{L^{2}(\Omega)}^{2}+\|\nabla \mu\|_{L^{2}(\Omega)}^{2}\right\}+\int_{\Omega}\left(|\nabla\phi_t  |^{2}+\rho| \phi_t  |^{2}+\rho|u_t  |^{2}\right)dx\\
			&\leqslant C (\epsilon_0+\|\nabla u\|_{L^{2}(\Omega)}^{2}+\|\nabla \mu\|_{L^{2}(\Omega)}^{2})
			(\epsilon_0\|\nabla u\|_{L^{2}(\Omega)}^{4}+\|\nabla\mu\|_{L^{2}(\Omega)}^{2}+\|\nabla u\|_{L^{2}(\Omega)}^{2})+C\|\nabla \mu\|_{L^{2}(\Omega)}^{2}.
		\end{split}
	\end{align}
	According to Lemma \ref{lem2-5}, one finds that
	\begin{align}\label{4.15}
		\begin{split}
			&\sup_{t\in[0,T]}(\|\nabla u\|_{L^{2}(\Omega)}^{2}+\|\nabla\mu\|_{L^{2}(\Omega)}^{2})\\
			&\leqslant C exp\left({\int_{0}^{T}C (\epsilon_0\|\nabla u\|_{L^{2}(\Omega)}^{4}+\|\nabla \mu\|_{L^{2}(\Omega)}^{2}+\|\nabla u\|_{L^{2}(\Omega)}^{2}) dt}\right)\cdot\\
			&\quad\left(\|\nabla u_0\|_{L^{2}(\Omega)}^{2}+\|\nabla\mu_0\|_{L^{2}(\Omega)}^{2}+\epsilon_0+\int_{0}^{T}\|\nabla\mu\|_{L^{2}(\Omega)}^{2} dt\right)\\
			&\leqslant C\epsilon_0exp\left({\int_{0}^{T}C\epsilon_0\|\nabla u\|_{L^{2}(\Omega)}^{4} dt}\right).
		\end{split}
	\end{align}
	
	Moreover, the existence of local strong solutions to the problem \eqref{1.1}-\eqref{1.4} in Lemma \ref{Existence}, implies that there exists $T_{*}>0$ such that the problem \eqref{1.1}-\eqref{1.4} admits a unique local strong solution $(\rho,u,\phi,\mu)$ on $(0,T_{*}).$ So, there exists a positive constant $M>0$ such that
	\begin{equation}\label{4.16}
		\int_{0}^{T}\|\nabla u\|_{L^{2}(\Omega)}^{4} dt\leqslant 2M,
	\end{equation}
	holds for any $T\in (0,T^*).$ Using \eqref{4.2} and \eqref{4.15}, one obtains that
	\begin{eqnarray*}
		&\int_{0}^{T}\|\nabla u\|_{L^{2}(\Omega)}^{4} dt&\leqslant \sup_{t\in [0,T]}\|\nabla u\|^{2}_{L^{2}(\Omega)}\int_{0}^{T}\|\nabla u\|_{L^{2}(\Omega)}^{2} dt\\
		&&\leqslant C\epsilon_0^2 exp\left({\int_{0}^{T}C\epsilon_0\|\nabla u\|_{L^{2}(\Omega)}^{4} dt}\right)\\
		&&\leqslant C\epsilon_{0}^{2}exp\left({2C\epsilon_0M}\right).
	\end{eqnarray*}
	The fact
	$$\lim\limits_{x\rightarrow 0+}x^2exp\left({ax}\right)=0$$
	for any $a>0,$ implies that there exists small enough $\epsilon_0$ such that
	$$a\epsilon_{0}^{2}exp\left({2a\epsilon_0M}\right)<M.$$ So, we have
	\begin{equation}\label{4.17}
		\int_{0}^{T}\|\nabla u\|_{L^{2}(\Omega)}^{4} dt\leqslant M.
	\end{equation}
	Thus, it is deduced from \eqref{4.16} and \eqref{4.17} that
	$$\sup_{t\in [0,T]}(\|\nabla u\|_{L^{2}(\Omega)}^{2}+\|\nabla \mu\|_{L^{2}(\Omega)}^{2})\leqslant C\epsilon_0.$$
	Therefore, with the help of Lemma \ref{lem4-1} and Lemma \ref{lem4-2}, one gets that
	\begin{align}\label{4.18}
		\begin{split}
			&\sup_{t\in [0,T]}\big(\|\nabla u\|_{L^{2}(\Omega)}^{2}+\|\nabla \mu\|_{L^{2}(\Omega)}^{2}+\|\phi\|_{W^{2,6}(\Omega)}^{2} \big)+ \int_{0}^{T}(\| \phi_t\|_{H^{1}(\Omega)}^{2}+\| u_t \|_{L^{2}(\Omega)}^{2}+\|u\|_{H^{2}(\Omega)}^{2}) dt\\
			&\leqslant C\epsilon_0.
		\end{split}
	\end{align}
	This completes the proof of Lemma \ref{lem4-4}.
\end{proof}
With all the a priori estimates in Lemma \ref{lem4-1}-Lemma \ref{lem4-4}, we are now in a position to prove Theorem \ref{thm1.2}.

\noindent{\bf Proof of Theorem \ref{thm1.2}.} According to Lemma \ref{Existence}, there exists a $T_*>0$ such that the problem \eqref{1.1}-\eqref{1.4} has a unique local strong solution $(\rho,u,\phi,\mu)$ on $\Omega\times(0,T^*).$ It follows from
\begin{equation}\label{4.19}
	T^{*} \triangleq \sup\{T \ |\ (\rho,u,\phi,\mu) \ \text{is a solution on} \ \Omega\times(0,T]\ \text{and} \ \eqref{4.18} \ \text{holds} \}.
\end{equation}
Hence, it follows from \eqref{4.18} and Lemma \ref{Existence} that \eqref{existence} holds for any $T\in (0,T^{*})$ with $T$ finite.
Now, we claim that
\begin{equation}\label{4.20}
	T^{*} = \infty.
\end{equation}

\noindent Otherwise $T^{*}<\infty$. Then Lemma \ref{lem4-1} and Lemma \ref{lem4-4} hold at $T=T^{*}$. So,
$$(\rho^{T},u^{T},\phi^{T},\mu^{T})\triangleq (\rho,u,\phi,\mu)(x,T^{*})=\lim_{t\to T^{*}}(\rho,u,\phi,\mu)(x,t)$$
satisfies
\begin{eqnarray*}
	\begin{cases}
		0<\frac{1}{C}\leqslant \rho^{T}\leqslant C\epsilon_0,\ u^{T}\in V_{\sigma}(\Omega),\phi^{T}\in H^{2}(\Omega), \partial_n\phi^T = 0,\\
		\mu^{T}=-\frac{\Delta\phi^T}{\rho^T}+\Psi^{\prime}(\phi^T)\in H^{1}(\Omega).
	\end{cases}
\end{eqnarray*}
Therefore, one takes $ (\rho^{T},u^{T},\phi^{T},\mu^{T})$ as the initial data and apply Lemma \ref{Existence} to extend the local strong solution beyond $T^{*}$. This contradicts the assumption of $T^{*}$ in \eqref{4.19}. Hence, \eqref{4.20} holds, which means that the global existence of strong solution to the problem \eqref{1.1}-\eqref{1.4} has been proven.

Moreover, one gets that
\begin{align}\label{4.21}
		\frac{d}{dt}\int_{\Omega}\big(\frac{\rho u^{2}}{2}+\frac{\rho(\phi^{2}-1)^{2}}{4}+\frac{|\nabla \phi|^{2}}{2}\big) dx+\int_{\Omega}\big(v(\phi)|\mathbb{D}u |^{2}+|\nabla\mu |^{2}\big) dx=0.
\end{align}
Using Lemma \ref{lem4-2}, one finds that
\begin{eqnarray}\label{4.22}
	\int_{\Omega}\frac{1}{2}\rho|u |^{2} dx\leqslant\frac{\epsilon_0}{2}\|u\|_{L^{2}(\Omega)}^{2}\leqslant\frac{\sqrt{2}\epsilon_0}{2v_*}\int_{\Omega}v(\phi)|\mathbb{D}u |^{2} dx.
\end{eqnarray}
Taking into account of \eqref{1.1}$_5$ and Lemma \ref{lem4-2}, one gets that
\begin{eqnarray*}
	&&\int_{\Omega}\rho (\phi^{2}-1)^{2} dx+\int_{\Omega}|\nabla\phi |^{2} dx
	=\int_{\Omega}\rho\mu\phi dx-\int_{\Omega}\rho(\phi^{2}-1) dx\\
	&&\leqslant\epsilon_0\|\mu\|_{L^{2}(\Omega)}\|\phi\|_{L^{2}(\Omega)}
	+\frac{1}{2}\int_{\Omega}\rho(\phi^{2}-1)^{2} dx+\frac{1}{2}\int_{\Omega}\rho  dx\\
	&&\leqslant C\epsilon_0^2\|\nabla \mu\|_{L^{2}(\Omega)} +\frac{1}{2}\int_{\Omega}\rho(\phi^{2}-1)^{2} dx+\frac{1}{2}\int_{\Omega}\rho  dx,
\end{eqnarray*}
which implies that
\begin{align}\label{4.23}
	\int_{\Omega}\frac{1}{4}\rho (\phi^{2}-1)^{2} dx+\int_{\Omega}\frac{1}{2}|\nabla\phi |^{2} dx
	\leqslant C\epsilon_0^2\|\nabla \mu\|_{L^{2}(\Omega)}+\frac{1}{4}\int_{\Omega}\rho  dx.
\end{align}
Adding \eqref{4.22} and \eqref{4.23}, one obtains that
\begin{eqnarray}\label{4.24}
	&&\int_{\Omega}\left(\frac{1}{2}\rho|u |^{2} dx+\frac{1}{4}\rho (\phi^{2}-1)^{2} dx+\frac{1}{2}|\nabla\phi |^{2}\right) dx\nonumber\\
	&&\leqslant \frac{\sqrt{2}\epsilon_0}{2v_*}\int_{\Omega}v(\phi)|\mathbb{D}u |^{2} dx+ C\epsilon_0^2\|\nabla \mu\|_{L^{2}(\Omega)}+\frac{1}{4}\int_{\Omega}\rho  dx\\
	&& \leqslant \max\{\frac{\sqrt{2}}{2v_*}, C\epsilon_0\}\epsilon_0\big(\int_{\Omega}v(\phi)|\mathbb{D}u |^{2} dx+\|\nabla \mu\|_{L^{2}(\Omega)}\big)
	+\frac{1}{4}\int_{\Omega}\rho  dx.\nonumber
\end{eqnarray}
Combining \eqref{4.21} and \eqref{4.24}, it is easy to get that
\begin{align}
	\begin{split}
		&\frac{d}{dt}\int_{\Omega}\big(\frac{\rho u^{2}}{2}+\frac{\rho(\phi^{2}-1)^{2}}{4}+\frac{|\nabla \phi|^{2}}{2}\big) dx+a_0\int_{\Omega}\big(\frac{1}{2}\rho|u |^{2} dx+\frac{1}{4}\rho (\phi^{2}-1)^{2} dx+\frac{1}{2}|\nabla\phi |^{2}\big) dx\nonumber\\
		&\leqslant\frac{1}{4}a_0\int_{\Omega}\rho  dx
		=\frac{1}{4}a_0\int_{\Omega}\rho_0  dx,
	\end{split}
\end{align}
where $a_0=\left(\max\{\frac{\sqrt{2}}{2v_*}, C\epsilon_0\}\epsilon_0\right)^{-1}.$  Hence, it follows from \eqref{4-0} that
\begin{align*}
	\begin{split}
		&\int_{\Omega}\big(\frac{\rho u^{2}}{2}+\frac{\rho(\phi^{2}-1)^{2}}{4}+\frac{|\nabla \phi|^{2}}{2}\big) dx\\
		&\leqslant \left(\int_{\Omega}\big(\frac{\rho_0 u_0^{2}}{2}+\frac{\rho_0(\phi_0^{2}-1)^{2}}{4}+\frac{\lvert\nabla \phi_0\lvert^{2}}{2}\big) dx -\frac{a_0}{4}\int_{\Omega}\rho_0  dx\right)exp\big(-a_0t\big)+\frac{a_0}{4}\int_{\Omega}\rho_0  dx\\
		&\leqslant C\epsilon_0exp\big(-a_0t\big)+\frac{1}{4}a_0\int_{\Omega}\rho_0  dx.
	\end{split}
\end{align*}
Therefore, the proof of Theorem \ref{thm1.2} is completed.



\begin{thebibliography}{00}
	
	\bibitem{2009-AH}
	Abels, Helmut, \textit{Existence of weak solutions for a diffuse interface model for viscous, incompressible fluids with general densities},
	Comm. Math. Phys. 289, no.1, 45-73, 2009.
	
	
	\bibitem{Abels-2009}
	Abels, Helmut, \textit{On a diffuse interface model for tow-phase flows of viscous, incompressible fluids with matched densities}, Arch. Ration. Mech. Anal. 194, no.2, 463-506, 2009.
	
	
	\bibitem{2012-AH}
	Abels, Helmut, \textit{Strong well-posedness of a diffuse interface model for a viscous, quasi-incompressible two-phase flow}, SIAM J. Math. Anal. 44, no.1, 316-340, 2012.
	
	
	\bibitem{2013-AH15}
	Abels, Helmut; Depner, Daniel; Garcke, Harald, \textit{Existence of weak solutions for a diffuse interface model for two-phase flows of incompressible fluids with different densities}, J. Math. Fluid Mech. 15, no.3, 453-480, 2013.
	
	\bibitem{2008-FE}
	Abels, Helmut; Feireisl, Eduard, \textit{On a diffuse interface model for two-phase flow of compressible viscous fluids}, Indiana Univ. Math. J. 57, no.2, 659-698, 2008.
	
	
	\bibitem{1998-AMW}
	Anderson, D. M.; McFadden, G. B.; Wheeler, A. A., \textit{Diffuse-interface methods in fluid mechanics}, Annual review of fluid mechanics, Vol. 30, 139-165, 1998.
	
	
	\bibitem{Boyer-1999}
	Boyer, Franck, \textit{Mathematical study of multi-phase flow under shear through order parameter formulation}, Asymptot. Anal. 20, no.2, 175-212, 1999.
	
	\bibitem{Beal-1984}
	Beale, J. T.; Kato, T.; Majda, A., \textit{Remarks on the breakdown of smooth solutions for the 3D Euler equations}, Comm. Math. Phys. 94, no.1, 61-66, 1984.
	
	\bibitem{Cho-2004}
	Cho, Yonggeun; Choe, Hi Jun; Kim, Hyunseok, \textit{Unique solvability of the initial boundary value problems for compressible viscous fluid}, J. Math. Pures Appl. 83(9), no.2, 243-275, 2004.
	
	
	\bibitem{2019-CL}
	Cherfils, Laurence; Feireisl, Eduard; Mich\'{a}lek, Martin; Miranville, Alain; Petcu, Madalina; Pra\v{z}\'{a}k, Dalibor, \textit{The compressible Navier-Stokes-Cahn-Hilliard
		equations with dynamic boundary conditions}, Math. Models Methods Appl. Sci. 29, no.14, 2557-2584, 2019.
	
	
	
	\bibitem{Cao-2012}
	Cao, Chongsheng; Gal, Ciprian G., \textit{Global solutions for the 2D NS-CH model for a two-phase flow of viscous, incompressible fluids with mixed partial viscosity and mobility}, Nonlinearity 25, no.11, 3211-3234, 2012.		
	
	
	\bibitem{1958-Cahn}
	Cahn J. W.; Hilliard J. E., \textit{Free energy of a nonuniform system. I. Interfacial free energy}, The Journal of chemical physics, 28(2), 258-267, 1958.
	
	
	
	\bibitem{Chen-He-Mei-Shi-2018}
	Chen, Yazhou; He, Qiaolin; Mei, Ming; Shi, Xiaoding, \textit{Asymptotic stability of solutions for 1-D compressible Navier-Stokes-Cahn-Hilliard system}, J. Math. Anal. Appl. 467, no.1, 185-206, 2018.
	
	
	\bibitem{Choe-2003}
	Choe, Hi Jun; Jin, Bum Ja, \textit{Regularity of weak solutions of the compressible Navier-Stokes equations}, J. Korean Math. Soc. 40, no.6, 1031-1050, 2003.
	
	\bibitem{Chella-1996}
	Chella R.; Vinals J., \textit{Mixing of a two-phase fluid by a cavity flow}, Physical Review E 53(4), 3832-3840, 1996.
	
	
	\bibitem{Doi-1997}
	Doi M., \textit{Dynamics of domains and textures, Theoretical Challenges in the Dynamics of Complex Fluids}, NATO ASI Series E Applied Sciences-Advanced Study Institute 339, 293-314, 1997.
	
	\bibitem{Gunton-1983}
	Domb, C.; Lebowitz, J. L., \textit{Phase transitions and critical phenomena. Vol. 8}, Academic Press, Inc., London, 1983.
	
	
	\bibitem{2010PDE}
	Evans, Lawrence C., \textit{Partial differential equations}, Grad. Stud. Math. 19, 2nd edition, 2010.
	
	
	\bibitem{2004Dynamics}
	Feireisl, Eduard, \textit{Dynamics of viscous compressible fluids}, Oxford Lecture Ser. Math. Appl. 26, 2004.
	
	
	\bibitem{Fan-2008}
	Fan, Jishan; Jiang, Song, \textit{Blow-up criteria for the Navier-Stokes equations of compressible fluids}, J. Hyperbolic Differ. Equ. 5, no.1, 167-185, 2008.
	
	
	
	
	\bibitem{Gal-2010}
	Gal, Ciprian G.; Grasselli, Maurizio, \textit{Asymptotic behavior of a Cahn-Hilliard-Navier-Stokes in 2D}, Ann. Inst. Henri Poincar\'{e}, Anal. Non Lin\'{e}aire 27, no.1, 401-436, 2010.
	
	
	\bibitem{Gal-2016}
	Gal, C. G.; Grasselli, M.; Miranville, A., \textit{Cahn-Hilliard-Navier-Stokes systems with moving contact lines}, Calc. Var. Partial Differential Equations 55, no.3, Art. 50, 47 pp, 2016.
	
	
	\bibitem{Gui-2018}
	Gui, Guilong; Li, Zhenbang, \textit{Global well-posedness of the 2-D incompressible Navier-Stokes-Cahn-Hilliard system with a singular free energy density}, 	arXiv:1810.12705, 2018.
	
	\bibitem{Giorgini-2019}
	Giorgini, Andrea; Miranville, Alain; Temam, Roger, \textit{Uniqueness and regularity for the Navier-Stokes-Cahn-Hilliard system}, SIAM J. Math. Anal. 51, no.3, 2535-2574, 2019.
	
	
	
	\bibitem{Gurtin-1996}
	Gurtin, Morton E.; Polignone, Debra; Vi\~{n}als, Jorge, \textit{Two-phase binary fluids and immiscible fluids described by an
		order parameter}, Math. Models Methods Appl. Sci. 6, no.6, 815-831, 1996.
	
	\bibitem{Giorgini-2020}
	Giorgini, Andrea; Temam, Roger, \textit{Weak and strong solutions to the nonhomogeneous incompressible Navier-Stokes-Cahn-Hilliard system}, J. Math. Pures Appl. 144(9), 194-249, 2020.
	
	
	\bibitem{Huang-2009}
	Huang, Xiangdi, \textit{Some results on blowup of solutions to the compressible Navier-Stokes equations}, Ph.D Thesis. The Chinese University of Hong Kong, ProQuest LLC, Ann Arbor, MI, 98 pp, 2009.
	
	\bibitem{Heida-2013}
	Heida, Martin, \textit{On the derivation of thermodynamically consistent boundary conditions for the Cahn-Hilliard-Navier-Stokes system}, Internat. J. Engrg. Sci. 62, 126-156, 2013.
	
	\bibitem{He-2021}
	He, Jingning, \textit{Global weak solutions to a Navier-Stokes-Cahn-Hilliard system with chemotaxis and
		singular potential}, Nonlinearity34, no.4, 2155-2190, 2021.
	
	\bibitem{Huang-2009-1}
	Huang, Xiangdi; Li, Jing, \textit{On breakdown of solutions to the full compressible Navier-Stokes equations}, Methods Appl. Anal. 16, no.4, 479-490, 2009.
	
	\bibitem{Huang-2011}
	Huang, Xiangdi; Li, Jing; Xin, Zhouping, \textit{Blowup criterion for the compressible flows with vacuum states}, Comm. Math. Phys. 301, no.1, 23-35, 2011.
	
	
	
	\bibitem{Heida-2012}
	Heida, Martin; M\'{a}lek, Josef; Rajagopal, K. R., \textit{On the development and generalizations of Cahn-Hilliard equations
		within a thermodynamic framework}, Z. Angew. Math. Phys. 63, no.1, 145-169, 2012.
	
	
	
	\bibitem{He-Wu-2021}
	He, Jingning; Wu, Hao, \textit{Global well-posedness of a Navier-Stokes-Cahn-Hilliard system with chemotaxis
		and singular potential in 2D}, J. Differential Equations 297, 47-80, 2021.
	
	
	
	
	
	
	\bibitem{Huang-2010}
	Huang, XiangDi; Xin, ZhouPing, \textit{A blow-up criterion for classical solutions to the compressible Navier-Stokes equations}, Sci. China Math. 53, no.3, 671-686, 2010.
	
	
	
	\bibitem{1968-Landau}
	Landau, L. D.; Lifshitz, E. M., \textit{Course of Theoretical Physics, Statistical Physics, Vol. 5: Statistical physics}, Pergamon Press, Oxford-Edinburgh-New York, 1968.
	
	
	\bibitem{1968Linear}
	Lady\v{z}enskaja, O. A.; Solonnikov, V. A.; Ura$\text{l}^{\prime}$ceva, N. N., \textit{Linear and Quasi-linear Equations of Parabolic Type}, American Mathematical Society, Providence, RI, 1968.
	
	
	\bibitem{Lam-Wu-2018}
	Lam, Kei Fong; Wu, Hao, \textit{Thermodynamically consistent Navier-Stokes-Cahn-Hilliard models with mass
		transfer and chemotaxis}, European J. Appl. Math. 29, no.4, 595-644, 2018.
	
	\bibitem{Qiu-Tang-Wang-2023}
	Qiu, Zhaoyang; Tang, Yanbin; Wang, Huaqiao, \textit{Well-posedness for the Cahn-Hilliard-Navier-Stokes equaiton with random initial data}, Potential Anal. 59, no.2, 753-770, 2023.
	
	
	\bibitem{Sun-Zhang-2018}
	Sun, Yongzhong; Zhang, Zhifei, \textit{Blow-up criteria of strong solutions and conditional regularity of weak solutions for the compressible Navier-Stokes equations}. Handbook of mathematical analysis in mechanics of viscous fluids, 2263-2324, 2018.
	
	
	
\end{thebibliography}

\section*{Acknowledgement}
The work of Li Fang was supported in part by the National Natural Science Foundation of China (Grant no. 11501445). The work of Zhenhua Guo was supported in part by the National Natural Science Foundation of China (Grant no. 11931013).

\end{document}